\documentclass[11pt,reqno]{amsart}
\usepackage[utf8x]{inputenc}
\usepackage{amssymb,amsmath,amsthm,amsfonts,amscd,mathtools}
\usepackage[normalem]{ulem}
\usepackage[colorlinks=true]{hyperref}
\hypersetup{colorlinks,allcolors=blue}
\usepackage{zref-clever} 
\usepackage{enumitem}
\usepackage[dvipsnames]{xcolor}
\usepackage{microtype}
\usepackage{tikz}
\usetikzlibrary{calc,angles,quotes,arrows.meta,decorations.markings}
\usepackage{pgfplots}
\usepackage{hyperref}
\usetikzlibrary{angles,quotes,calc}
\usepackage{pgfplots}
\pgfplotsset{compat=1.15}
\usepackage{mathrsfs}
\usetikzlibrary{arrows}
\definecolor{cqcqcq}{rgb}{0.75,0.75,0.75}
\definecolor{Amber}{RGB}{255,191,0}
\usepackage{microtype}

\usepackage{graphicx}

\usepackage{caption}
\usepackage{subcaption}

\theoremstyle{definition}
\newtheorem*{ack}{Acknowledgements}

\let\oldnewtheorem\newtheorem

\RenewDocumentCommand{\newtheorem}{momo}{
  \IfValueTF{#2}{
    \AddToHook{env/#1/begin}{
      \zcsetup{countertype={#2=#1}}}
      \zcRefTypeSetup{#1}{
Name-sg = #3 ,
      }
    \oldnewtheorem{#1}[#2]{#3}
  }{
    \AddToHook{env/#1/begin}{
      \zcsetup{countertype={#1=#1}}}
    \zcRefTypeSetup{#1}{
Name-sg = #3 ,
      }
    \IfValueTF{#4}{
      \oldnewtheorem{#1}{#3}[#4]
    }{
      \oldnewtheorem{#1}{#3}
    }
  }
}

\newcommand{\Cref}[1]{\zcref[S]{#1}}

\theoremstyle{definition}
\newtheorem{dfn}{Definition}[section]

\theoremstyle{plain}
\newtheorem{thm}[dfn]{Theorem}
\newtheorem{pop}[dfn]{Proposition}
\newtheorem{lem}[dfn]{Lemma}
\newtheorem{cor}[dfn]{Corollary}

\theoremstyle{remark}
\newtheorem{rem}[dfn]{Remark}

\newcommand{\R}{\mathbb R}

\newcommand{\LpLS}{Lorentzian pre-length space }
\newcommand{\LpLSn}{Lorentzian pre-length space}

\newcommand{\ma}{\measuredangle}
\newcommand{\tma}{\widetilde \measuredangle}
\newcommand{\bp}{\bar p}
\newcommand{\ba}{\bar a}
\newcommand{\bb}{\bar b}
\newcommand{\bc}{\bar c}
\newcommand{\bz}{\bar z}

\newcommand{\bx}{\bar x}

\DeclareMathOperator{\conv}{conv}
\DeclareMathOperator{\arccosh}{arcosh}

\newcommand{\lm}[1]{\mathbb{L}^2(#1)}

\usepackage[foot]{amsaddr}

\title[Splitting Theorem for non-positively curved Lorentzian spaces]{A Splitting Theorem for non-positively curved Lorentzian spaces}

\author[Barton]{Joe Barton $^\mathrm{A}$}
\email[Barton]{joe.barton@univie.ac.at}

\author[Beran]{Tobias Beran $^\mathrm{A}$}
\email[Beran]{tobias.beran@univie.ac.at}

\author[Che]{Mauricio Che $^\mathrm{A}$}
\email[Che]{mauricio.adrian.che.moguel@univie.ac.at}

\author[Gieger]{Sebastian Gieger $^\mathrm{A}$}
\email[Gieger]{sebastian.gieger@univie.ac.at}

\author[R\"ohrig]{Jona R\"ohrig $^\mathrm{A}$}
\email[R\"ohrig]{jona.roehrig@univie.ac.at}

\address{$^\mathrm{A}$Department of Mathematics, University of Vienna, Oskar-Morgenstern-Platz 1, 1090 Wien, Austria}

\author[Rott]{Felix Rott $^\mathrm{B}$}
\email[Rott]{frott@sissa.it}

\address{$^\mathrm{B}$SISSA, Via Bonomea 265, 34136 Trieste, Italy}

\date{\today}

\begin{document}

\begin{abstract}
We prove a splitting theorem for Lorentzian pre-length spaces with global non-positive timelike curvature. 
Additionally, we extend the first variation formula to spaces with any timelike curvature bound, either from above or below, and different from 0.
\bigskip

\noindent
\emph{Keywords:} Metric geometry, Lorentzian geometry, Lorentzian length spaces, splitting theorems, first variation formula, curvature bounds
\medskip

\noindent
\emph{MSC2020:}
53C23, 
51K10, 
53C50, 
53B30 
\end{abstract}

\maketitle
\tableofcontents

\section{Introduction}
Splitting theorems are prototypical examples of rigidity results about spaces satisfying extremal geometric conditions. In synthetic approaches to curvature bounds, the presence of such results is often taken as an indicator of a robust structural theory, closer in spirit to the classical Riemannian and Lorentzian settings. 
The classical splitting theorem from Riemannian geometry was developed by Toponogov in \cite{Top59, Top64} and the moral argument for this result and many of its successors goes as follows: a non-negative lower bound on the curvature usually yields conjugate points, beyond which geodesics stop being distance-realising. A geodesic defined on $\R$ in a non-negatively curved space is therefore special, and can occur only under very restrictive circumstances.

Slightly more technically, often it is shown that the \emph{Busemann functions} associated with a straight line are \emph{affine}; this property turns out to completely characterise when a complete Riemannian manifold splits, even without any assumption on the curvature \cite{Inn82}.
In positive signature, splitting theorems exist in the settings of Riemannian manifolds (first under sectional, and later under Ricci curvature bounds), Alexandrov spaces, and even RCD spaces \cite{CG71, Gig14,Mil67}. 
In Lorentzian signature, analogues exist for all of these settings except, so far, for the case of RCD spaces \cite{BEMG85b, BEMG85a, BORS23,  Esc88, Gal89, New90}. 

A crucial point common to all of the results mentioned above is the presence of a \emph{lower} curvature bound. 
By contrast, there are seemingly fewer splitting results for spaces with an upper curvature bound, at least partially because the `classical splitting theorem' is simply wrong in this setting: neither hyperbolic space nor de Sitter space are isometric to a product of another space with the real line (they do, however, split as warped products). 
In fact, in spaces with upper curvature bounds the absence of conjugate points along geodesics is no longer noteworthy, so the above mentioned intuitive argument no longer works. 
As a consequence, splitting results in this context typically require additional assumptions (e.g.\ of topological or group-theoretic nature \cite{burago-burago-ivanov2001, GW71, Inn82, LY72,  Sch85}). 
Alternatively, such spaces may split as cones, which is a special type of warped product \cite{AB05}. 

The main goal of this work is to establish a Lorentzian analogue of a splitting theorem for CAT($0$) spaces proved in \cite{burago-burago-ivanov2001}, where it is assumed that the space is a union of lines. 
The setting of our main theorem (and of the preparatory results leading to it) is that of Lorentzian pre-length spaces, a novel approach to non-smooth Lorentzian geometry. 
Introduced in \cite{kunzinger-saemann2018}, building on earlier work by Alexander--Bishop \cite{alexander-bishop2008} and Harris \cite{harris1982}, this field has demonstrated an impressive growth over the last several years, and is now an established area at the intersection of metric geometry, Lorentzian geometry, mathematical general relativity and optimal transport. 
A simplified version of our main result can be stated as follows. 

\begin{thm}\label{thm:main}
Let $X$ be a \LpLS with timelike curvature globally bounded above by $0$. Let $\gamma$ be a complete timelike line and define
\[
S=\{\alpha:\alpha\text{ is a complete timelike line parallel to }\gamma\}\,.
\]
Then $S$ admits a metric that makes it into a $\mathrm{CAT}(0)$ space and the set $\bigcup_{\alpha\in S}\alpha(\R)$ is isometric to the Lorentzian product $\prescript{-}{}{\R}\times S$. 
In particular, if $X=\bigcup_{\alpha\in S}\alpha(\R)$, then the splitting is global.
\end{thm}

For the precise statement, see \Cref{thm:lorentzianHadamard}. Concerning the outline of the proof, there are some similarities with its lower curvature counterparts, yet also a few different techniques tailored more towards the nature of spaces with upper curvature bounds. 
As to the step from positive signature to the Lorentzian one, we were able to adapt most of the proof of \cite{burago-burago-ivanov2001}, with the usual attention paid to keeping the causality in check. 
In particular, an important preliminary result concerns the rigid behaviour of comparison configurations where equality is achieved in the defining inequality for upper curvature bounds (\Cref{pop:triangleCBequality}). 
We also make use of the first variation formula (\Cref{fvf}), which in the synthetic Lorentzian context, so far, has only been introduced for non-negative lower curvature bounds 
\cite{BarreraMontesdeOcaSolis2022}. 
Moreover, we observe that the main reason why spaces with an upper curvature bound may fail to split is that (future and past) rays do not fit together to a line (\Cref{cor:concat0angle}). 
More precisely, to any point $p$ not on a given line $\gamma$ we may construct two parallel rays associated to $\gamma$ that start at $p$. In Lorentzian signature, these can be naturally described as future and past rays; an analogous construction exists in positive signature, obtained by respectively considering a limit procedure with the two different `ends' of the line. The angle between these rays, however, is generally different from $0$ (as is the case, for instance, in de Sitter space) and hence, they do not form a line. 
In fact, this is the only obstruction to them concatenating to a line. 

\begin{ack}
We would to thank Michael Kunzinger and Roland Steinbauer for their helpful discussions in the initial stages of this project. 

JB, TB, MC, SG, and JR were funded by the Austrian Science Fund (FWF) [Grant DOI's: 10.55776/EFP6, 10.55776/STA32]. 

FR acknowledges the support of the European Union - NextGenerationEU, in the framework of the PRIN Project `Contemporary perspectives on geometry and gravity' (code 2022JJ8KER – CUP G53D23001810006). The views and opinions expressed are solely those of the authors and do not necessarily reflect those of the European Union, nor can the European Union be held responsible for them.

For open access purposes, the authors have applied a CC BY public copyright license to any author-accepted manuscript version arising from this submission.
\end{ack}

\section{Preliminaries}\label{s:preliminaries}
In this section, we summarise notation, definitions and general results about Lorentzian pre-length spaces that are fundamental for this work. We refer the reader to \cite{kunzinger-saemann2018} for the general theory of Lorentzian pre-length spaces.

Throughout this paper, we will denote a \LpLS $(X,d,\leq,\ll,\tau)$ by $X$. Causal (resp.\ timelike) curves in $X$ are monotone curves with respect to $\leq$ (resp.\ $\ll$). A future-directed \textit{geodesic} in $X$ is a causal curve $\gamma \colon I\to X$, defined on some interval $I\subseteq \mathbb{R}$, such that
\[
\tau(\gamma(s),\gamma(t)) = L_\tau(\gamma|_{[s,t]})
\]
for all $s\leq t$, $s,t\in I$, where $L_{\tau}$ is the $\tau$-length functional. Past-directed geodesics are defined similarly. Moreover, any timelike geodesic admits constant-speed parametrizations, i.e.\ one may assume that for some $C > 0$, 
\[
\tau(\gamma(s),\gamma(t)) = C(t-s) 
\] 
for all $s\leq t$. In particular, when $C=1$, we say that $\gamma$ is \textit{parametrized by $\tau$-arclength}. 
We refer to a curve that locally realises the $\tau$-distance as a \emph{local geodesic}.

Furthermore, whenever $x,y\in X$ are chronologically related, we denote by $[x,y]$ the geodesic  
starting at  $x$ and ending at $y$, if it exists and is unique, or if the choice is fixed, or if this choice is not relevant. 
We denote by $\lm{K}$ the Lorentzian model space of constant curvature $K$ and by $D_K$ its (finite) diameter, cf.\ \cite[Definition 4.5]{kunzinger-saemann2018} and \cite[Definition 2.6]{BNR25}. 
Finally, recall that $X$ is called \emph{regular} if any geodesic between timelike related points is timelike, i.e. does not contain a null segment.

\begin{dfn}[Timelike triangles]
Let $X$ be a \LpLSn. An \emph{(unordered) timelike triangle} $\triangle(x,y,z)$ is a collection of pairwise timelike related points $x,y,z$ (not necessarily in the order $x\ll y\ll z$) together with geodesics $[x,y],[x,z],[y,z]$ connecting them. A \emph{comparison triangle} in $\lm{K}$ for $\triangle(x,y,z)$ is a timelike triangle $\triangle(\bar{x},\bar{y},\bar{z})$ in $\lm{K}$ with the same sidelengths as $\triangle(x,y,z)$.
We say a timelike triangle \emph{satisfies size bounds} if it possess a comparison triangle, which is then unique up to isometry.
This is the case if and only if the longest side is less than $D_K$, cf.\ \cite[Lemma 4.6]{kunzinger-saemann2018}. 
\end{dfn}

\begin{dfn}[Timelike curvature bounds]
\label{def: triangle comparison}
Let $X$ be a \LpLSn. An open subset $U$ is called a \emph{$(\geq K)$- (resp.\ $(\leq K)$-)comparison neighbourhood} (in the sense of triangle comparison) if:
\begin{enumerate}[label=(\roman*)]
\item $\tau$ is continuous on $(U\times U) \cap \tau^{-1}([0,D_K))$, and this set is open. 
\item For all $x \ll y$ with $\tau(x,y) < D_K$, there exists a geodesic between them that is contained in $U$.
\item \label{TLCB.item3} Let $\triangle (x,y,z)$ be a timelike triangle in $U$, with $p,q$ two points on the sides of $\triangle (x,y,z)$. Let $\triangle(\bar{x}, \bar{y}, \bar{z})$ be a comparison triangle in $\lm{K}$ for $\triangle (x,y,z)$ and $\bar{p},\bar{q}$ comparison points for $p$ and $q$, respectively. Then 
\begin{equation}
\label{eq: timelike triangle comparison inequality}
\tau(p,q) \leq \tau(\bar{p},\bar{q}) \quad \text{ (resp.\ } \tau(p,q) \geq \tau(\bar{p}, \bar{q}) \text{)} \, .
\end{equation}
\end{enumerate}
We say $X$ has \emph{timelike curvature bounded below (resp.\ above) by $K$} if $X$ can be covered by $(\geq K)$- (resp.\ $(\leq K)$-)comparison neighbourhoods. 
We say $X$ has timelike curvature \emph{globally} bounded below (resp.\ above) by $K$ if $X$ itself is a $(\geq K)$- (resp.\ $(\leq K)$-)comparison neighbourhood. 
\end{dfn}
Note that within a $(\geq K)$-comparison neighbourhood, $p \ll q$ implies $\bar{p} \ll \bar{q}$, and within a $(\leq K)$-comparison neighbourhood, $\bar{p} \ll \bar{q}$ implies $p \ll q$. 
Furthermore, recall that within a $(\leq K)$-comparison neighbourhood, geodesics are unique, cf.\ \cite[Theorem 4.7]{BNR25}. 
We refer to \cite{beran-kunzinger-rott2024} for equivalent formulations of \Cref{def: triangle comparison}. 

In \cite{beran-saemann2023}, a notion of angles between geodesics that have a common endpoint is introduced. 
Namely, for any timelike geodesics $\alpha\colon [0,a)\to X$, $\beta\colon [0,b)\to X$ such that $\alpha(0)=\beta(0)=p$, 
\begin{equation}
\label{eq:angle between curves}
\ma_p(\alpha,\beta) = \limsup_{s,t\to 0} \widetilde{\ma}^K_p(\alpha(s),\beta(t)) \, ,
\end{equation}
where $\widetilde{\ma}^K_p(\alpha(s),\beta(t))$ denotes the angle at the vertex corresponding to $p$ in a comparison triangle for $\triangle(p,\alpha(s),\beta(t))$ in $\mathbb{L}^2(K)$ and the $\limsup$ only runs over the $s,t$ where $\widetilde{\ma}^K_p(\alpha(s),\beta(t))$ is defined. 
This definition of angle is independent of $K$, cf.\ \cite[Proposition 2.14]{beran-saemann2023}. 
If $K=0$, we may drop the $K$ on comparison angles and just write $\tma$. 

If $X$ has a curvature bound from either direction, then the limit superior in \eqref{eq:angle between curves} is actually a limit, cf.\ \cite[Lemma~4.10 and Theorem~4.13]{beran-saemann2023}. 
Moreover, if $X$ has timelike curvature bounded above by $K$, then the angle between any geodesics with the same time orientation is finite, and if $X$ has timelike curvature bounded below by $K$, then the angle between any geodesics with different time orientations is finite.

We also define the \textit{signed angles}: 
\begin{equation}\label{eq:signed angle}
\ma_p^S(\alpha,\beta) = \begin{cases}
-\ma_p(\alpha,\beta) & \text{if $\alpha$ and $\beta$ have the same time orientation,} \\
\ma_p(\alpha,\beta) & \text{otherwise}.
\end{cases}
\end{equation}

Furthermore, by \cite[Corollary~4.11]{beran-saemann2023}, if $X$ has timelike curvature bounded above by $K$, then
\begin{equation}\label{eq:angle condition}
\widetilde{\ma}_p^{K,S}(\alpha(s),\beta(t)) \leq \ma_p^S(\alpha,\beta)\, ,
\end{equation}
whenever both angles are well-defined. 
If for $y,z \in I^{\pm}(x)$ the choice of geodesics $[x,y],[x,z]$ does not matter, we may use the shorthand notation $\ma_x(y,z)=\ma_x([x,y],[x,z])$. 

Recall that angles between geodesics satisfy the triangle inequality in the following sense (see \cite[Theorems~3.1 and 4.5]{beran-saemann2023} and \cite[Lemma~4.14]{beran-kunzinger-rott2024}).

\begin{pop}[Triangle inequality for angles]
\label{pop:triangle inequality angles}
Let $X$ be a strongly causal Lorentzian pre-length space with $\tau$ locally finite valued and locally continuous, and let $\alpha,\beta,\gamma\colon [0, \varepsilon) \to X$ be timelike curves starting at $x:=\alpha(0) = \beta(0) = \gamma(0)$. 
\begin{enumerate}[label=(\roman*)]
    \item\label{FuFuFu/PaPaPa} If $\alpha,\beta,\gamma$ are all future-directed or all past-directed, then
    \[
    \ma_x(\alpha,\gamma) \leq \ma_x(\alpha,\beta) + \ma_x(\beta,\gamma) \, .
    \]
    \item\label{FuFuPa/PaPaFu} Assume that $X$ is strongly causal, and $\tau$ is locally finite and locally continuous. If $\alpha,\beta$ are both future-directed, $\gamma$ is past-directed, and $\ma_x(\alpha,\gamma)$ exists, then 
    \[
    \ma_{x}(\alpha,\gamma)\leq \ma_{x}(\alpha,\beta) + \ma_x(\beta,\gamma)\, .    
    \]
    Analogously for $\alpha,\beta$ both past-directed and $\gamma$ is future-directed.
    \item\label{triIneqAng:AlongGeo} If $X$ has timelike curvature bounded above by $K$ for some $K\in \R$, $\alpha,\beta$ are future-directed, $\gamma$ is past-directed, and the concatenation of $\gamma\cdot \beta$ is a geodesic, then
    \[
    \ma_{x}(\alpha,\gamma)\leq \ma_{x}(\alpha,\beta) \, .    
    \]
    An analogous inequality holds if $\alpha,\beta$ are both past-directed and $\gamma$ is future-directed such that $\beta\cdot \gamma$ is a geodesic.
\end{enumerate}
\end{pop}

The following angle sum formula for timelike triangles in Minkowski space can be regarded as an analogue to the fact that the three interior angles of a Euclidean triangle sum up to $\pi$. We include the proof for the sake of completeness.

\begin{lem}[Angle sum for triangles in Minkowski space]
\label{lem:angleSumMink}
Let $a\ll b\ll c$ be points in Minkowski space. Then
\[
\ma_a(b,c)+\ma_c(a,b)=\ma_b(a,c)
\]
or equivalently, with signed angles,
\[
\ma_a^S(b,c)+\ma_b^S(a,c)+\ma_c^S(a,b)=0 \, .
\]
\end{lem}
\begin{proof}
Assume that the vector $c-a$ is vertical. Then shifting a copy of $\triangle(a,b,c)$ vertically upward by $\tau(a,c)$ makes a new triangle $\triangle(a_2,b_2,c_2)$ with $a_2=c$. Note that $\tau(b,b_2)=\tau(a,c)$, thus also $\triangle(b,c,b_2)$ is a time reversed copy of $\triangle(a,b,c)$. By angle additivity in the plane, we get that $\ma_b(a,c)=\ma_c(b,b_2)=\ma_c(a,b)+\ma_c(b_2,c_2)=\ma_c(a,b)+\ma_a(b,c)$ as desired.
\end{proof}

We now define a notion of isomorphism between Lorentzian pre-length spaces. 

\begin{dfn}[Structure preserving maps]
\label{def:isometry}
Given two Lorentzian pre-length spaces $X$ and $Y$, we say a map $f\colon X\to Y$ is \textit{causality-preserving} if $f(x)\leq_Y f(y)$ if and only if $x\leq_X y$. On the other hand, $f$ is \textit{$\tau$-preserving} if $\tau_Y(f(x),f(y))=\tau_X(x,y)$ for any $x,y\in X$.

Furthermore, an \textit{isometry} between Lorentzian pre-length spaces is a bijective map $f\colon X\to Y$ such that $f$ and thus also $f^{-1}$ is $\tau$- and causality-preserving. 
\end{dfn}

\begin{rem}[Isometries, topology and causality]
Note that a $\tau$-preserving map need not be causality-preserving, cf.\ \cite[Remark 3.2.6]{BR24} (even though the terminology in this reference is slightly different, the remark is still valid). 

Further observe that if $X$ and $Y$ are strongly causal Lorentzian pre-length spaces, then any isometry $f\colon X\to Y$ in the sense of \Cref{def:isometry} is automatically a homeomorphism. This is because, in strongly causal Lorentzian pre-length spaces, the metric topology coincides with the Alexandrov topology, and an isometry in the sense of \Cref{def:isometry} preserves the subbase of the Alexandrov topology given by chronological diamonds. 
\end{rem}

The following notion of Lorentzian product appeared in \cite[Definition 3.1]{BORS23}. 
If the base metric space is geodesic, then Lorentzian products in that sense are a special case of the even earlier introduced \emph{generalised cones} in \cite{alexander-graf-kunzinger-saemann2023} (by setting the warping function $f \equiv 1$), cf.\ \cite[Proposition~3.6]{Ber25}. 

\begin{dfn}[Lorentzian product]
\label{def: lor product}
Let $(X,d)$ be a metric space. The \textit{Lorentzian product} $Y= \R \times X$ is given by the product metric 
\begin{equation*}
D\colon Y \times Y \to \R, \ D((s,x),(t,y)):=\sqrt{|t-s|^2+d(x,y)^2} \, ;    
\end{equation*}
the chronological relation 
\[(s,x) \ll (t,y)\quad \text{if and only if}\quad t-s > d(x,y) \, ;\] the causal relation \[(s,x) \leq (t,y) \quad \text{if and only if}\quad t-s \geq d(x,y) \, ;\] and the time separation $\tau\colon Y \times Y \to \R$ given by 
\begin{equation*}
\tau((s,x),(t,y)):=\sqrt{(t-s)^2 - d(x,y)^2}    
\end{equation*}
if $(s,x) \leq (t,y)$, and $0$ otherwise. 

When it is clear that the Lorentzian product is meant, we simply denote it by $\prescript{-}{}{\R}\times X$.
\end{dfn}

The following proposition is a consequence of results and arguments from \cite{alexander-graf-kunzinger-saemann2023}. We include a proof for the sake of completeness.

\begin{pop}[Properties of Lorentzian products]
\label{pop:AGKS-products}
Let $X$ be a length space, then the Lorentzian product $Y=\prescript{-}{}{\mathbb{R}}\times X$ is a Lorentzian pre-length space. Furthermore, if $X$ is geodesic and $Y$ has timelike curvature (globally) bounded above by $0$, then $X$ has non-positive curvature in the triangle comparison sense (globally, i.e.\ $X$ is a $\mathrm{CAT}(0)$ space).
\end{pop}

\begin{proof}
The fact that $Y$ is a Lorentzian pre-length space follows from \cite[Proposition~3.2]{BORS23}.
The fact that $X$ has non-positive curvature whenever $Y$ has timelike curvature bounded above by $0$ follows from \cite[Theorem~5.7]{alexander-graf-kunzinger-saemann2023}.

Furthermore, by carefully inspecting the proof of \cite[Lemma~5.6]{alexander-graf-kunzinger-saemann2023}, one can see that when $Y$ has timelike curvature globally bounded above by $0$, the constant $C>0$ corresponding to the $(\leq 0)$-comparison neighbourhood $Y$ given by that Lemma is $C=\infty$. This implies that in the proof of \cite[Theorem~5.7]{alexander-graf-kunzinger-saemann2023} one can set $U=X$ and obtain that $X$ is a $(\leq 0)$-comparison neighbourhood in the metric sense, i.e.\ $X$ is a $\mathrm{CAT}(0)$ space. 
\end{proof}

\section{First variation formula}

In this section, we prove a first variation formula for Lorentzian pre-length spaces with timelike curvature bounded above or below, analogous to the positive-signature one (see \cite[Theorem~4.5.6]{burago-burago-ivanov2001} and \cite[Theorem~3.5]{otsu-shioya1994}). To the best of our knowledge, the only prior versions of the first variation formula in synthetic Lorentzian geometry are \cite[Theorems~19 and 20]{BarreraMontesdeOcaSolis2022}, which hold for Lorentzian pre-length spaces with timelike non-negative curvature.  

Throughout this section, we denote by $\tau_s(p,q):=\max\{\tau(p,q),\tau(q,p)\}$ the order independent time separation function.
We start by considering the model spaces.

\begin{pop}[First variation formula in the model space]
\label{pop: First variation formula in the model space}
Let $K\in\mathbb R$ and $a,b \in \mathbb L^2(K)$ be timelike related. 
Let $\beta$ be a geodesic joining the two points and $\gamma:[0,\delta]\to X$ be a future-directed timelike geodesic parametrized by arclength starting at $\gamma(0)= a$. Assume $c:=\gamma(\delta)$ is also timelike related to $b$ and the triangle $\triangle (b,a,c)$ satisfies size bounds. For $t\in[0,\delta]$, we call side lengths $y:=\tau_s(b,a)$, $t=\tau(a,\gamma(t))$ and $z(t):=\tau_s(b,\gamma(t))$. Then
\begin{equation}\label{eq:constantfirstvariation}
\sigma\cosh(\measuredangle^S_a(\gamma,\beta))=\lim_{t\to0}\frac{z(t)-y}{t} \, ,
\end{equation}
where 
\[\sigma:=\begin{cases}
+1 & \text{if }\ b\ll a \, ,\\
-1 & \text{if }\ b\gg a \, .
\end{cases}\ \]
\end{pop}

\begin{proof}
    First of all, note that by the definition of comparison angles, we have $\measuredangle^S_a(\gamma,\alpha)=\widetilde\measuredangle_a^{K,S}(b,c)=:\theta$. 
    Recall the Lorentzian law of cosines for triangles in spaces of constant sectional curvature $K\in\mathbb R$:
    \[\begin{cases}
        y^2+t^2=z^2-2\sigma yt\cosh(\theta) 
        \\
        \cos(\sqrt{|K|}z)=\cos(\sqrt{|K|}y)\cos(\sqrt{|K|}t)-\sigma\cosh(\theta)\sin(\sqrt{|K|}y)\sin(\sqrt{|K|}t) 
        \\
        \cosh(\sqrt{K}z)=\cosh(\sqrt{K}y)\cosh(\sqrt{K}t)+\sigma\cosh(\theta)\sinh(\sqrt{K}y)\sinh(\sqrt{K}t) 
    \end{cases}\]
    where the first equation is used for $K=0$, the second one or $K<0$ and the final one for $K>0$.
    The proof is now a simple computation. We show \eqref{eq:constantfirstvariation} only in the case $K=0$, with the other two cases working analogously by applying the appropriate version of the Lorentzian law of cosines.
    \begin{align*}
    \lim_{t\to0}\frac{z(t)-y}{t} 
    &=  \lim_{t\to0} \frac{\sqrt{t^2 +y^2+2\sigma ty\cosh{\theta}}-y}{t}\\
    &= \lim_{t\to0}\frac{t+ \sigma y\cosh\theta}{  \sqrt{t^2 +y^2+2\sigma t y \cosh{\theta}}}\\
    &=\frac{\sigma y \cosh\theta}{\sqrt{y^2}} \\
    &=\sigma \cosh\theta \, , 
    \end{align*}
    where the second equation follows from L'Hôpital's rule. 
    
\end{proof}
We now extend this result to general \LpLSn s. This will be done in two parts. 
\begin{pop}[First variation formula, first inequality]
\label{prop:geq}
Let $X$ be a \LpLSn. 
Let $U \subseteq X$ be a ($\leq K$) or a ($\geq K$)-curvature comparison neighbourhood, $\gamma:[0,\delta]\to U$ be a future-directed timelike geodesic parametrized by arclength, and $p$ be timelike related to $a:=\gamma(0)$. Assume further that there exists a geodesic $\beta_t$ from $p$ to $\gamma(t)$ for all $t\in [0,\delta)$. Define $l(t) = \tau_s(p,\gamma(t))$. If the signed angle $\theta:=\ma^S_a (\gamma,\beta_0)$ exists and is finite then
\[ \liminf_{t\searrow 0}\frac{l(t)-l(0)}{t} \geq \sigma\cosh(\ma_a(\gamma,\beta_0)) \, ,\]
where 
\[\sigma:=\begin{cases}
    +1 & p\ll a \, ,\\
    -1 & p\gg a \, .
\end{cases}\ \]
\end{pop}

\begin{figure}
    \centering
    \begin{tikzpicture}[line cap=round,line join=round,>=triangle 45,x=0.3cm,y=0.3cm]
\clip(-11.,-17.) rectangle (5.,7.);
\draw [shift={(18.,0.)},line width=1.pt]  plot[domain=2.819842099193151:3.4633432079864352,variable=\t]({1.*18.973665961010276*cos(\t r)+0.*18.973665961010276*sin(\t r)},{0.*18.973665961010276*cos(\t r)+1.*18.973665961010276*sin(\t r)});
\draw [shift={(14.,-30.)},line width=1.pt]  plot[domain=2.0988707752212563:2.613518205163434,variable=\t]({1.*27.78488797889961*cos(\t r)+0.*27.78488797889961*sin(\t r)},{0.*27.78488797889961*cos(\t r)+1.*27.78488797889961*sin(\t r)});
\draw [shift={(23.39336333724209,-21.83353214142168)},line width=1.pt]  plot[domain=2.3683699585604665:2.9686464691206,variable=\t]({1.*33.92071115815807*cos(\t r)+0.*33.92071115815807*sin(\t r)},{0.*33.92071115815807*cos(\t r)+1.*33.92071115815807*sin(\t r)});
\draw [shift={(37.698458999321275,-18.29891124365339)},line width=1.pt]  plot[domain=2.6601357550221714:2.941817930083284,variable=\t]({1.*43.5292648941362*cos(\t r)+0.*43.5292648941362*sin(\t r)},{0.*43.5292648941362*cos(\t r)+1.*43.5292648941362*sin(\t r)});
\draw [shift={(0.,-6.)},line width=1.pt]  plot[domain=1.839098465228527:3.705627898214046,variable=\t]({1.*2.0272592363457687*cos(\t r)+0.*2.0272592363457687*sin(\t r)},{0.*2.0272592363457687*cos(\t r)+1.*2.0272592363457687*sin(\t r)});
\begin{scriptsize}
\draw [fill=black] (0.,-6.) circle (1.5pt);
\draw[color=black] (0.7,-6) node {$a$};
\draw[color=black] (-1,5) node {$\gamma$};
\draw [fill=black] (-10.,-16.) circle (1.5pt);
\draw[color=black] (-9,-16.22840161113279) node {$p$};
\draw[color=black] (-7,-13) node {$\beta_0$};
\draw [fill=black] (-0.8824520041607258,1.8582266574792823) circle (1.5pt);
\draw[color=black] (2,2) node {$c_n=\gamma(t_n)$};
\draw[color=black] (-7.8,-6.172531674089621) node {$\beta_{t_n}$};
\draw [fill=black] (-4.935474214766206,-9.666583753291132) circle (1.5pt);
\draw[color=black] (-4,-10) node {$b_m$};
\draw[color=black] (-2,-3) node {$z_{mn}$};
\draw[color=black] (0,-1.6069355532481824) node {$t_n$};
\draw[color=black] (-1.8,-8.2) node {$y_m$};
\draw[color=black] (-1.1,-5.7) node {$\theta$};
\end{scriptsize}
\end{tikzpicture}
\hspace{1cm}
\begin{tikzpicture}[line cap=round,line join=round,>=triangle 45,x=0.3cm,y=0.3cm]
\clip(20.,-19.) rectangle (27.,4.);
\draw [line width=1.pt] (24.,2.)-- (26.,-8.);
\draw [line width=1.pt] (26.,-8.)-- (22.,-12.);
\draw [line width=1.pt] (22.,-12.)-- (24.,2.);
\draw [shift={(26.,-8.)},line width=1.pt]  plot[domain=1.7681918866447768:3.9269908169872414,variable=\t]({1.*2.5*cos(\t r)+0.*2.5*sin(\t r)},{0.*2.5*cos(\t r)+1.*2.5*sin(\t r)});
\begin{scriptsize}
\draw [fill=black] (24.,2.) circle (1.5pt);
\draw[color=black] (24.282688811099185,2.7162370773349602) node {$\tilde c_n$};
\draw [fill=black] (26.,-8.) circle (1.5pt);
\draw[color=black] (26.7,-8) node {$\tilde a$};
\draw[color=black] (25.8,-2.7780311595755767) node {$t_n$};
\draw [fill=black] (22.,-12.) circle (1.5pt);
\draw[color=black] (21.7,-12.4) node {$\tilde b_m$};
\draw[color=black] (25,-10.4) node {$y_m$};
\draw[color=black] (21.8,-4.740269815615054) node {$z_{mn}$};
\draw[color=black] (24.7,-7.605138253432691) node {$\widetilde\theta_{mn}$};
\end{scriptsize}
\end{tikzpicture}
    \caption{Set up for the first variation formula for $\sigma=1$.}
    \label{fig:fvf}
\end{figure}
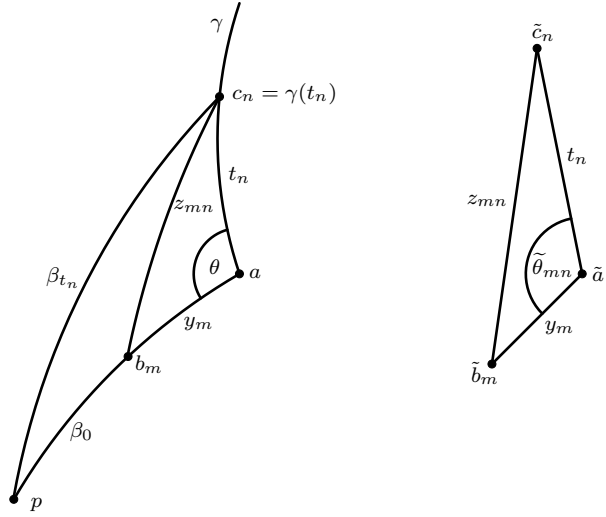

\begin{proof}
    We have already established the formula (with equality) in the comparison spaces in \Cref{pop: First variation formula in the model space}. Indeed, take timelike triangles $\triangle(\tilde b, \tilde a, \tilde c)$ in $\mathbb{L}^2(K)$, satisfying size bounds, with side lengths $t:=\tau(\tilde a,\tilde c)$, $y:=\tau_s(\tilde b,\tilde a)$ and angle $\widetilde \theta = \widetilde\measuredangle^{K,S}_{\tilde a}(\tilde b,\tilde c)$. Then, for every $\varepsilon>0$ and for any $\widetilde\theta\in\mathbb R$, there exists a $T_{\varepsilon,\widetilde\theta}>0$ such that for all $t\in(0,T_{\varepsilon,\widetilde\theta})$, we have
    \begin{equation}\label{eq:cosh}
    \left|\sigma \cosh{\widetilde\theta} - \frac{z(t)-y}{t}\right|<\varepsilon \, ,
    \end{equation}
    where $z(t)$ is the length of the remaining side of the triangle, uniquely determined by the other three parameters $t,y$ and $\widetilde\theta$.\\
    
    We first consider the case $\sigma=1$ and $U$ is a $(\leq K)$-curvature comparison neighbourhood. If we restrict ourselves to all angles $\widetilde\theta$ in the compact interval $[0,\theta]$, one can even find a $T_\varepsilon$ independent of $\widetilde\theta$ such that \eqref{eq:cosh} is satisfied for all $t\in(0,T_\varepsilon)$. This is true because the mapping $\widetilde\theta\mapsto T_{\varepsilon,\widetilde\theta}$ can be chosen to be continuous. As a result, if we take any positive sequence $(t_n)_{n\in\mathbb{N}}$ with $\lim_{n\to\infty}t_n=0$ and any monotonically increasing sequence $(\widetilde\theta_n)_{n\in\mathbb{N}}$ satisfying $0\leq \widetilde \theta_n\leq \theta\in\R$, then there exists an $N_\varepsilon\in\mathbb N$ such that for $n\geq N_\varepsilon$ we have $t_n\leq T_\varepsilon$ and consequently
    \[\left|\cosh{\widetilde\theta_n} - \frac{z(t_n,\widetilde\theta_n)-y}{t_n}\right|<\varepsilon \, .\]
    In particular, there exists a $\widetilde\theta_\infty\in\mathbb R$ such that 
    \[   \cosh{\widetilde\theta_\infty} = \lim_{n\to \infty} \cosh{\widetilde\theta_n} = \lim_{n\to\infty} \frac{z(t_n,\widetilde\theta_n)-y}{t_n} \, .\]
    Here $z(t_n,\widetilde\theta_n)$ is the length of the third side of the triangle with side lengths $t_n$ and $y$ and angle $\widetilde\theta_n$ in $\mathbb L^2(K)$.\\

    So far, all considerations have only taken place in the comparison space. We now return to the Lorentzian pre-length space X. Take any sequence $b_m \in \beta_0$ converging to $a$ and assume, without loss of generality, that all $b_m$ lie in the comparison neighbourhood $U$.  
    Define $c_n:=\gamma(t_n)$ and consider the triangles $\triangle (b_m, a, c_n)$ with comparison angle $\widetilde\theta_{mn}=  \widetilde\measuredangle^{K,S}_a (b_m, c_n)$ and side lengths $y_m$, $t_n$ and $z_{mn}$, see \Cref{fig:fvf}.
    By the curvature bound, $\widetilde\theta_{mn}$ is monotonically increasing in $n$, and so by the argument above, for any fixed $m\in \mathbb{N}$, we obtain
    \[\lim_{n\to \infty} \cosh{\widetilde\theta_{mn}} = \lim_{n\to\infty} \frac{z_{mn}-y_m}{t_n} \, .\]
    By the reverse triangle inequality, we can now estimate
    \[z_{mn}-y_m=\tau(p,b_m)+z_{mn}-(\tau(p,b_m)+y_m)\leq l(t_n)-l(0) \, .\]
    Note that this estimate is independent of $m$. Hence altogether, we have
    \[\cosh\theta =\lim_{m\to\infty}\lim_{n\to\infty}\cosh \widetilde\theta_{mn}=\lim_{m\to\infty}\lim_{n\to\infty}\frac{z_{mn}-y_m}{t_n} \leq \liminf_{n\to\infty} \frac{l(t_n)-l(0)}{t_n} \, .\]

    If $\sigma=-1$, this argument becomes slightly more subtle. After choosing a sequence $(b_m)_{m\in\mathbb N}$ converging to $a$, the comparison angles $\widetilde\theta_{mn}=  \widetilde\measuredangle^{K,S}_a (b_m, c_n)$ will once again be monotonically increasing and bounded from above by $\theta$. However, this means that $\widetilde\theta_{mn}$ will lie in the non-compact interval $(-\infty,\theta]$. Consequently, there will not be a $T_\varepsilon>0$ independent of $\widetilde\theta\in(-\infty,\theta]$ such that \eqref{eq:cosh} holds. However, it is enough to consider angles in the compact interval $[\theta_0,\theta]$, where $\theta_0:=\widetilde\measuredangle_a^{K,S}(b_0,\gamma(\delta))$, since by the curvature bound, $\widetilde\theta_{nm}\geq\theta_0$ for all $n,m\in\mathbb N$. $T_\varepsilon$ will then also depend on $b_0$, but this is irrelevant for the rest of the proof, which works completely analogously to the case $\sigma=1$. Indeed, in the end, we obtain
    \[-\cosh\theta\leq\liminf_{n\to\infty}\frac{l(t_n)-l(0)}{t_n} \, .\]

Finally, let $U$ be a $(\geq K)$-curvature comparison neighbourhood. The comparison angles $\widetilde\theta_{mn}$ are now no longer monotonically increasing and bounded above by $\theta$, but are monotonically decreasing and bounded below by $\theta$. Since we only used the curvature bound to deduce the monotonicity and subsequent convergence of this sequence, the proof for $\geq K$ works almost analogously, with some small technical adaptations.
\end{proof}

To complete the proof of the first variation formula, that is, to show equality in \eqref{eq:constantfirstvariation} for general spaces, we need to prove the reverse inequality: 
\[\limsup_{t\searrow 0}\frac{l(t)-l(0)}{t}\leq \sigma\cosh(\ma_a(\gamma,\beta_0)) \, .\]
In order to achieve this, we first prove the following technical lemma in the comparison space $\mathbb L^2(K)$.

\begin{lem}[First variation formula, second inequality in the model space]
\label{lem:firstvargeq}
     Let $K\in\mathbb R$ and $\triangle(a, b, c)$ be a timelike triangle in $\mathbb L^2(K)$ satisfying size bounds such that either $b\ll a\ll c$ or $a\ll c\ll b$. Moreover, let $\theta:=\ma^S_a(b,c)$. Calling the side lengths of the triangle $t:=\tau_s(a,c),\ z:=\tau_s(b,c)$ and $y:=\tau_s(b,a)$, we obtain
    \begin{equation*}
        \sigma\cosh(\theta)\geq\frac{z-y}{t} \, ,
    \end{equation*}
    where 
    \[\sigma:=\begin{cases}
    +1 & b\ll a\ll c,\\
    -1 & a\ll c\ll b.
\end{cases}\ \]
\end{lem}

\begin{proof}
    We first show the proof for $K=0$ and then for general $K\in\R$.
    By the Lorentzian law of cosines, we have
    \begin{align*}
        \sigma \cosh(\theta)&=\frac{z^2-y^2-t^2}{2yt}\\
        &=\frac{z-y}{t} \frac{z+y}{2y}-\frac{t}{2y} - \frac{z-y}{t} + \frac{z-y}{t}\\
        &=\underbrace{\sigma\frac{z-y}{t}}_{\geq1}\underbrace{\sigma\left(\frac{z+y}{2y}-1\right)}_{=\sigma\frac{z-y}{2y}\geq\frac{t}{2y}}-\frac{t}{2y}+\frac{z-y}{t}\\
        &\geq\frac{t}{2y}-\frac{t}{2y}+\frac{z-y}{t}=\frac{z-y}{t} \, ,
    \end{align*}
    where we used multiple times that $\sigma(z-y)\geq t$ by the reverse triangle inequality. \\

    For $K=1$ and $\sigma=+1$, we must prove
    \begin{equation*}
        \frac{z-y}{t} \leq 
        \sigma \cosh(\theta) =\frac{\cosh(z) - \cosh(t)\cosh(y)}{\sinh(t) \sinh(y)} 
        =\frac{\cosh(z) - \cosh(t+y)}{\sinh(t) \sinh(y)} +1 \, , 
    \end{equation*}
    where we have used the addition theorem for $\cosh(t+y)$. By the reverse triangle inequality, $z-(t+y)\geq 0$, and hence, the above is equivalent to 
    \begin{equation}\label{eq:equivalentformulation}
    \frac{\cosh(z)-\cosh(t+y)}{z-(t+y)} \geq \frac{\sinh(t)\sinh(y)}{t} \, .
    \end{equation}
    We observe that the left hand side is a difference quotient of $\cosh$. The mean value theorem states that there is a $\lambda\in[t+y,z]$ such that 
    \[\sinh(\lambda)=\frac{d}{d s}\bigg|_{s=\lambda}\cosh(s) = \frac{\cosh(z)-\cosh(t+y)}{z-(t+y)} \, .\]
    Since $\sinh$ is monotonically increasing in $\R$ and $\lambda\geq t+y$, to prove \eqref{eq:equivalentformulation}, it is enough to show
    \begin{equation}\label{eq:almostequivalent}
    \sinh(t+y) \geq \frac{\sinh(t)\sinh(y)}{t} \, .
    \end{equation}
    By the addition theorem for $\sinh(t+y)$ and the fact that for $t>0$ $\frac{\sinh(t)}{t}\leq \cosh(t)$, we see
    \[\sinh(t+y)=\sinh(t)\cosh(y)+\cosh(t)\sinh(y)\geq\frac{\sinh(t)\sinh(y)}{t}\]
    proving \eqref{eq:almostequivalent} and subsequently \eqref{eq:equivalentformulation}. 

    In fact, the equality holds for arbitrary $K>0$ by substituting $\sqrt{K}t$, $\sqrt{K}y$ and $\sqrt{K}z$ for $t$, $y$ and $z$ respectively:
    \[ \sigma \cosh(\theta) =\frac{\cosh(\sqrt{K}z) - \cosh(\sqrt{K} t)\cosh(\sqrt{K} y)}{\sinh(\sqrt{K} t) \sinh(\sqrt{K} y)} \geq \frac{\sqrt{K} z - \sqrt{K} y}{\sqrt{K} t} = \frac{z-y}{t} \, .\]

    The other three cases ($K=\pm1, \sigma=\pm1$) work similarly by using different inequalities.
    Instead of repeating the computations, we only mention the differences to the previous case. 
    
    For $K=1,\sigma=-1$, we initially use the addition theorem for $\cosh(y-t)$ instead of $\cosh(y+t)$ and apply the reverse triangle inequality $y\geq t+z$. At the end, instead of $\frac{\sinh(t)}{t}\leq\cosh(t)$, we use $\frac{\sinh(t)}{t}\geq 1$, which then yields the same result. 
    In the case $K=-1$, the inequality we need to prove reads 
    \[\frac{\cos(t)\cos(y) - \cos(z)}{\sin(t) \sin(y)} \geq \frac{z-y}{t} \, .\]
    We again distinguish the cases $\sigma=\pm1$. At the beginning, if $\sigma=+1$, we apply the addition theorem to $\cos(t+y)$ while for $\sigma=-1$ we apply it to $\cos(y-t)$. At the end, if $\sigma=+1$, we use that $\frac{\sin(t)}{t}\leq 1$ and if $\sigma=-1$, we use that $\frac{\sin(t)}{t}\geq \cos(t)$ for $t\in[0,\pi]$.  

    Finally, for arbitrary $K\in\R$, we note that in the case $K>0$ and $\sigma=+1$, the inequalities do not change when scaling all the side lengths by $\sqrt{|K|}$, which finishes the proof.
\end{proof}

Before we can proceed with showing the first variation formula, we state a lemma that will be needed in its proof.

\begin{lem}[Gauss--Bonnet formula for timelike triangles]
\label{gbtt}
    Let $(M,g)$ be a simply connected, two-dimensional Lorentzian $C^2$-manifold and denote the sectional curvature by the continuous function $K : M \to \mathbb{R}$ (with the convention that on de Sitter space we have $K \equiv +1$). 
    Given a geodesic timelike triangle $\triangle (a,b,c)$, $a \ll b \ll c$, with angles $\theta_1,\theta_2,\theta_3$ and interior $D \subseteq M$, the following special case of the Gauss--Bonnet formula holds:
    \begin{equation}
        \theta_1 + \theta_2 + \theta_3 = \int_D K \, \mathrm{d}\mathrm{A} \, .
    \end{equation}
\end{lem}
\begin{figure}
\begin{tikzpicture}[scale=0.6,>=Stealth, every node/.style={font=\small}]
  \coordinate (A) at (0,3);
  \coordinate (B) at (2,0);
  \coordinate (C) at (0,-3);
  \fill (A) circle (1.8pt) node[left=2pt] {$c$};
  \fill (B) circle (1.8pt) node[right=2pt]  {$b$};
  \fill (C) circle (1.8pt) node[right=2pt] {$a$};

  \tikzset{midarrow/.style={
    postaction={decorate},
    decoration={
      markings,
      mark=at position 0.5 with {\arrow{>}}
    }
  }}

  \draw[thick, midarrow] (A) -- (B);
  \draw[thick, midarrow] (B) -- (C);
  \draw[thick, midarrow] (C) -- (A);

  \def\ext{0.38}

  \coordinate (Aext) at ($(A) + 0.17*(A)-0.17*(C)$);
  \coordinate (Bext) at ($(B) + \ext*(B)-\ext*(A)$);
  \coordinate (Cext) at ($(C) + 0.29*(C)-0.29*(B)$);
  \pic[draw, angle radius=6mm, angle eccentricity=1.2]
       {angle = B--A--Aext};
  \pic[draw, angle radius=8mm, angle eccentricity=1.2]
       {angle = C--B--Bext};
  \pic[draw, angle radius=6mm, angle eccentricity=1.2]
       {angle = A--C--Cext};

  \node at ($(A)+(0.55,0.1)$) {$\hat{\theta}_3$};
  \node at ($(B)+(0.05,-0.75)$) {$\hat{\theta}_2$};
  \node at ($(C)+(-0.4,0.3)$) {$\hat{\theta}_1$};

  \pic[draw, angle radius=10mm, angle eccentricity=1.2]
       {angle = C--A--B};
  \pic[draw, angle radius=8mm, angle eccentricity=1.2]
       {angle = A--B--C};
  \pic[draw, angle radius=10mm, angle eccentricity=1.2]
       {angle = B--C--A};

  \node at ($(A)+(0.35,-1)$) {$\theta_3$};
  \node at ($(B)+(-0.6,0)$) {$\theta_2$};
  \node at ($(C)+(0.36,1.2)$) {$\theta_1$};

  \draw[dashed,->] (A) -- (Aext);
  \draw[dashed,->] (B) -- (Bext);
  \draw[dashed,->] (C) -- (Cext);
\end{tikzpicture}
\caption{The triangle considered as a closed loop, and the outer angles.}
\label{fig:placeholder}
\end{figure}
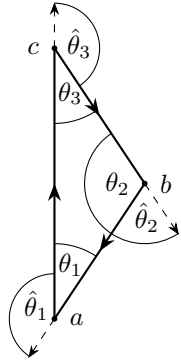
\begin{proof}
    The proof of the Lorentzian Gauss--Bonnet theorem works completely analogously to the Riemannian version. 
    We will not give a detailed proof here, but instead use the results of \cite{jee1984} and translate them into our setting. 
    Objects appearing in the notation of \cite{jee1984}
    will be denoted with a hat.

    The universal Gauss--Bonnet formula in the mentioned paper reads
    \begin{equation*}
        \int_D K\, \mathrm{d}\hat{\mathrm{A}}
        + \int_{\partial D} k_g\, \mathrm{d}l 
        + \sum_j \hat{\theta}_j 
        = 2\pi \, .
    \end{equation*}
    Here $k_g$ denotes the geodesic curvature of the boundary; it vanishes for us because all edges of the triangle are geodesics.

    The volume element $\mathrm{d}\hat{\mathrm{A}}$ is defined in coordinates exactly as in the Riemannian case, namely 
    $\mathrm{d}\hat{\mathrm{A}} = \sqrt{\det g}\, \mathrm{d}x_1 \wedge \mathrm{d}x_2$.
    Since $g$ is Lorentzian, $\det g < 0$, so the real-valued volume element is 
    $\mathrm{dA} = i\, \mathrm{d}\hat{\mathrm{A}}$.

    The quantities $\hat{\theta}_j$ are the outer angles at the corners of $D$, see \Cref{fig:placeholder}.  
    These are not the hyperbolic angles we use, but rather the analytic continuation of Riemannian angles into the complex plane.  
    The angle between two vectors is defined by
    \[
        \hat{\theta}(v_1,v_2) = \arccos \left( \frac{g(v_1,v_2)}{|v_1|_g\, |v_2|_g} \right) \, ,
        \qquad |v|_g = \sqrt{g(v,v)} \, .
    \]
    For timelike vectors, $|v|_g$ is purely imaginary, $|v_1|_g|v_2|_g$ is real and negative, and 
    \[
        \left|\frac{g(v_1,v_2)}{|v_1|_g |v_2|_g}\right| > 1 \, ,
    \]
    so $\arccos$ takes imaginary values.

    If two timelike vectors point in the same time direction (as for the outer angle $\hat{\theta}_2$ at $b$), then 
    $\frac{g(v_1,v_2)}{|v_1|_g|v_2|_g} > 1$, and the angle is required to lie in $i\mathbb{R}_{<0}$ (to fix the branch of $\arccos$).
    Comparing this with our definition of the hyperbolic angle at $b$, we obtain
    \begin{equation}\label{eq:GB1}
        \theta_2 
        = \arccosh\!\left( \frac{g(-v_1,v_2)}{|{-v_1}|_g\,|v_2|_g} \right) 
        = -i\, \hat{\theta}_2 \, .
    \end{equation}

    For the angles $\theta_1$ and $\theta_3$ at the future and past vertices $a$ and $c$, the velocity vectors $v_1$ and $v_2$ are timelike but have opposite time orientations.  
    In the notation of \cite{jee1984}, 
    \[
        \frac{g(v_1,v_2)}{|v_1|_g |v_2|_g} < -1 \, ,
    \]
    and the angles $\hat{\theta}_j$ lie in $\pi + i\mathbb{R}_{>0}$.  
    In our hyperbolic convention this yields
    \begin{equation}\label{eq:GB2}
        \theta_j
        = -\,\arccosh\!\left( -\frac{g(-v_1,v_2)}{|-v_1|_g\, |v_2|_g} \right)
        = i\pi - i\, \hat{\theta}_j \, .
    \end{equation}

    Summing \eqref{eq:GB1} and \eqref{eq:GB2}, we obtain
    \begin{equation*}
        \theta_1 + \theta_2 + \theta_3
        = 2i\pi - i \sum_j \hat{\theta}_j
        = i \int_D {K}\, \mathrm{d}\hat{\mathrm{A}}
        = \int_D K \, \mathrm{dA} \, ,
    \end{equation*}
    as claimed.
\end{proof}

\begin{rem}[Semi-continuity of angles and angle equality along geodesic]
\label{rem:CBBfact}
    In the following proof of the first variation formula we will be using two facts for angles in curvature comparison neighbourhoods:
    \begin{enumerate}[label=(\roman*)]
        \item Let $X$ be regular and let $U$ be a $(\geq K)$- (resp. $(\leq K)$-) comparison neighbourhood. Let $\gamma_n,\beta_n$ be sequences of timelike geodesics in $U$, sharing one endpoint $x_n$ and converging pointwise to timelike geodesics $\gamma,\beta$ in $U$. Then
        \begin{equation*}
        \begin{split}
            \limsup_{n\to\infty}\ma^S_{x_n}(\gamma_n,\beta_n)&\leq\ma_x^S(\gamma,\beta)\\
            (\text{resp. }\liminf_{n\to\infty}\ma_{x_n}^S(\gamma_n,\beta_n)&\geq\ma_x^S(\gamma,\beta)\ )
        \end{split}
        \end{equation*}
        (see \cite[Prop. 4.14]{beran-saemann2023}). Note that in the reference, this is only proved for geodesics $\gamma_n,\beta_n$ sharing the constant endpoint $x$ and not for a varying endpoint $x_n$ as we need it. The proof can, however, be easily adapted to the case where $x_n$ is not fixed. Also, note that in the reference, the space is assumed to be \emph{locally strictly timelike geodesically connected}, which follows from $X$ being regular and the whole setup being inside a comparison neighbourhood. \label{CBBfact1}
        \item Let $X$ be locally causally closed and let $U$ be a $(\geq K)$-comparison neighbourhood. Let $\gamma:[0,1]\to U$ be a future-directed timelike geodesic and let $\beta:[0,1]\to U$ be a timelike geodesic such that $\beta(0)=\gamma(\lambda)$ for some $\lambda\in(0,1)$. Then 
        \begin{equation*}
            \ma_{\beta(0)}(\beta,\gamma\big|_{[0,\lambda]})=\ma_{\beta(0)}(\beta,\gamma\big|_{[\lambda,1]})
        \end{equation*}
        (see \cite[Cor. 4.7]{beran-saemann2023}). \label{CBBfact2}
    \end{enumerate}
\end{rem}

\begin{thm}[First variation formula]
\label{fvf}
Let $X$ be a strongly causal, locally causally closed and regular Lorentzian pre-length space. Let $U\subseteq X$ be a ($\leq K$) or a ($\geq K$)-curvature comparison neighbourhood, let $\gamma:[0,\delta]\to U$ be a future-directed timelike geodesic parametrized by arclength, and let $p\in X $ be timelike related to $a:=\gamma(0)$. 
Let $t_n \to 0$ as $n \to \infty, t_n \geq 0$ such that a sequence $\beta_{t_n}$ of timelike geodesics from $p$ to $\gamma(t_n)$ converges uniformly to some $\beta_0$. 
Define $l(t) = \tau_s(p,\gamma(t))$. If the signed angle $\ma^S_a (\gamma,\beta_0)$ exists and is finite, then
\[ \lim_{n\to\infty}\frac{l(t_n)-l(0)}{t_n} = \sigma\cosh(\ma_a(\gamma,\beta_0)) \, ,\]
where 
\[\sigma:=\begin{cases}
+1 & p\ll a,\\
-1 & p\gg a.
\end{cases}\ \]
\end{thm}

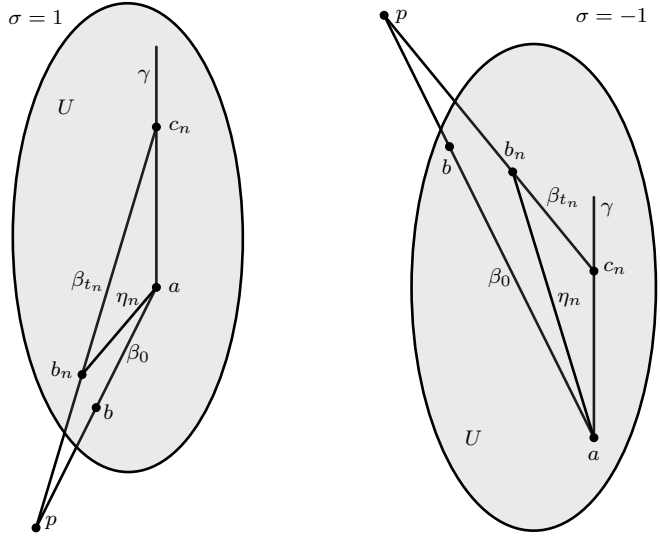
\begin{figure}
\begin{tikzpicture}[line cap=round,line join=round,>=triangle 45,x=0.8cm,y=0.8cm]
\clip(-1.,-0.5) rectangle (4.,9.);
\draw [line width=1.pt] (0.,0.)-- (2.,4.);
\draw [line width=1.pt] (2.,4.)-- (2.,8.);
\draw [line width=1.pt] (0.,0.)-- (2.,6.666666666666666);
\draw [fill=black!40,fill opacity=0.2, rotate around={-89.84216078629612:(1.525120716442052,4.826505919154599)},line width=1.pt] (1.525120716442052,4.826505919154599) ellipse (3.8981730992229418 and 1.9113146769959797);
\draw [line width=1.pt] (0.7645785596787179,2.5485951989290596)-- (2.,4.);
\begin{scriptsize}
\draw [fill=black] (0.,0.) circle (1.5pt);
\draw[color=black] (0.2522537454217871,0.08601238370045172) node {$p$};
\draw [fill=black] (2.,4.) circle (1.5pt);
\draw[color=black] (2.3,4) node {$a$};
\draw[color=black] (1.7,2.9) node {$\beta_0$};
\draw[color=black] (1.8,7.5) node {$\gamma$};
\draw [fill=black] (2.,6.666666666666666) circle (1.5pt);
\draw[color=black] (2.4,6.666) node {$c_n$};
\draw[color=black] (0.87,4.101158931992319) node {$\beta_{t_n}$};
\draw [fill=black] (0.7645785596787179,2.5485951989290596) circle (1.5pt);
\draw[color=black] (0.44,2.6317463257409823) node {$b_n$};
\draw [fill=black] (1.,2.) circle (1.5pt);
\draw[color=black] (1.2,1.9297976284871596) node {$b$};
\draw[color=black] (0.5,7) node {$U$};
\draw[color=black] (0,8.5) node {$\sigma=1$};
\draw[color=black] (1.54,3.8) node {$\eta_n$};
\end{scriptsize}
\end{tikzpicture}
\hspace{1cm}
\begin{tikzpicture}[line cap=round,line join=round,>=triangle 45,x=0.8cm,y=0.8cm]
\clip(9.3,-1.) rectangle (14.3,8.5);
\draw [line width=1.pt] (13.,1.)-- (13.,5.);
\draw [line width=1.pt] (13.,1.)-- (9.513672717571392,8.02451838319535);
\draw [line width=1.pt] (9.513672717571392,8.02451838319535)-- (13.,3.770810611393447);
\draw [fill=black!40, fill opacity=0.2, rotate around={90.:(12.,3.5)},line width=1.pt] (12.,3.5) ellipse (4.0464021784814825 and 2.030608428530651);
\draw [line width=1.pt] (11.65,5.42)-- (13.,1.);
\begin{scriptsize}
\draw [fill=black] (13.,1.) circle (1.5pt);
\draw[color=black] (13,0.7) node {$a$};
\draw[color=black] (13.22413374688328,4.8) node {$\gamma$};
\draw [fill=black] (9.513672717571392,8.02451838319535) circle (1.5pt);
\draw[color=black] (9.8,8) node {$p$};
\draw[color=black] (11.43870624620753,3.652815909255363) node {$\beta_0$};
\draw [fill=black] (13.,3.770810611393447) circle (1.5pt);
\draw[color=black] (13.35,3.827003958101778) node {$c_n$};
\draw[color=black] (12.5,5) node {$\beta_{t_n}$};
\draw [fill=black] (10.6,5.84) circle (1.5pt);
\draw[color=black] (10.55,5.5) node {$b$};
\draw [fill=black] (11.65,5.42) circle (1.5pt);
\draw[color=black] (11.7,5.78) node {$b_n$};
\draw[color=black] (11,1) node {$U$};
\draw[color=black] (13.3,8) node {$\sigma=-1$};
\draw[color=black] (12.6,3.3) node {$\eta_n$};
\end{scriptsize}
\end{tikzpicture}
\caption{Setup of the proof depending on $\sigma$.}
\label{fig: fvf setup}
\end{figure}

\begin{proof}
    By \Cref{prop:geq} it remains to be shown that 
    \begin{equation*}
        \limsup_{n\to\infty}\frac{l(t_n)-l(0)}{t}\leq \sigma\cosh(\ma_a(\gamma,\beta_0)) \, .
    \end{equation*}
    Assume that all curves $\beta_{t_n}$ as well as $\beta_0$ are defined on $[0,1]$ and choose $\varepsilon>0$ such that $b:=\beta_0(1-\varepsilon)$ lies in $U$. Take points $b_n$ on $\beta_{t_n}$ such that they converge to $b$ and, for all $n$, $b_n$ lies in $U$. We claim that 
    the sequence $b_n$ can, in particular, be chosen such that $b\ll b_n$ and $b_n\ll a$ if $\sigma=1$, as well as $b_n \ll b$ and $a \ll b_n$ if $\sigma=-1$ for all $n\in\mathbb N$ large enough.
    
    To see this, define the sets $U_{s,i}:=I(\beta_0(s-1/i),\beta_0(s+1/i))$, which, by strong causality, for sufficiently large $i$ (depending on $s$), form a neighbourhood basis for each $\beta_0(s)$ with $s\in(0,1)$. 
    For every $n$, we let $i(n)$ be the maximal natural number $i$ such that $\beta_{t_n}(s)\in U_{s,i}$ for all $s\in[1/i,1-1/i]$. Note that, by uniform convergence, $i(n)\to\infty$ as $n\to\infty$. We define $b_n:=\beta_{t_n}(1-\varepsilon+\sigma/i(n))$. Clearly $b_n\to b=\beta_0(1-\varepsilon)$. Moreover, for $n$ large enough and assuming $\sigma=1$, 
    \[b_n=\beta_{t_n}(1-\varepsilon+1/i(n))\in U_{1-\varepsilon+1/i(n),i(n)}\subseteq I^+(\beta_0(1-\varepsilon))=I^+(b) \, ,\]
    hence $b_n\gg b$ as desired. The fact that eventually $b_n\ll a$ is now just a consequence of the fact that $b\ll a$ and $\ll$ being an open relation. 
    Similarly, we show $b_n\ll b$ and $a\ll b_n$ in the case of $\sigma=-1$. 
    This proves the claim. 
    
    Setting $c_n := \gamma(t_n)$, by the reverse triangle inequality, we have 
    \begin{equation*}
        l(0)\geq\tau_s(p,b_n)+\tau_s(b_n,a)\geq\tau_s(p,b_n) + \tau_s(b_n,c_n)-\sigma t_n\cosh(\widetilde\ma_a^{S,K}(b_n,c_n)),
    \end{equation*}
    where for the second inequality we used \Cref{lem:firstvargeq}. Since $\tau_s(p,b_n)+\tau_s(b_n,c_n)=\tau_s(p,c_n)=l(t_n)$, we get
    \begin{equation}\label{eq:lhs}
        \frac{l(t_n)-l(0)}{t_n}\leq\sigma \cosh(\widetilde\ma_a^{S,K}(b_n,c_n)) \, .
    \end{equation}
    
    Let us now assume that $U$ is a $(\leq K)$-curvature comparison neighbourhood. First, recall that by the triangle inequality for angles, we have
    \begin{equation*}
        \ma_a(\gamma,\beta_0)\geq\ma_a(\gamma,\eta_n)-\ma_a(\eta_n,\beta_0)\, ,
    \end{equation*}
    and
    \begin{equation*}
        \ma_a(\gamma,\beta_0)\leq\ma_a(\gamma,\eta_n)+\ma_a(\eta_n,\beta_0) \, ,
    \end{equation*}
    where $\eta_n$ is the unique timelike geodesic from $b_n$ to $a$, see \Cref{fig: fvf setup}. Applying this to signed angles, where the first inequality is used for $\sigma=1$ and the second one for $\sigma=-1$, yields
    \begin{equation*}
        \sigma\cosh(\ma_a^S(\gamma,\beta_0))\geq\sigma\cosh(\ma_a^S(\gamma,\eta_n)+\ma_a^S(\eta_n,\beta_0)) \, .
    \end{equation*}
    Next we claim that the geodesics $\eta_n$ converge pointwise to $\beta_0\big|_{[1-\varepsilon,1]}$. To see this, first reparametrize all these curves so that they are defined on $[0,1]$ and for clarity write $\hat\beta_0$ instead of $\beta_0\big|_{[1-\varepsilon,1]}$. Choose $s\in(0,1)$, then it suffices to show that for any $\delta>0$ and $n\in\mathbb N$ large enough, we have $\eta_n(s)\in I(\hat\beta_0(s-\delta),\hat\beta_0(s+\delta))$. By the first part of the proof, we can consider the timelike triangles $\triangle(a,b_n,b)$ and their comparison triangle $\triangle(\tilde a,\tilde b_n,\tilde b)$ in $\mathbb L^2(K)$. Since $b_n\to b$ as $n\to\infty$ and the time separation function on $U$ is continuous, we have 
    \[\tau_s(b,b_n)+\tau_s(b_n,a)\to\tau_s(b,a)\hspace{1cm}(n\to\infty)\]
    so the side $[\tilde b_n,\tilde a]$ converges to the side $[\tilde b,\tilde a]$ in the comparison space. In particular, for $n$ large enough, the comparison point $[\tilde b_n,\tilde a](s)$ is timelike related to both $[\tilde b,\tilde a](s-\delta)$ and $[\tilde b,\tilde a](s+\delta)$. 
    By the upper curvature bound, we must now also have $\eta_n(s)\in I(\hat\beta_0(s-\delta),\hat\beta_0(s+\delta))$, as desired. 
    As a result, we can use semi-continuity of angles to obtain
    \begin{equation*}
        0=\ma_a^S(\beta_0,\beta_0)\leq\liminf_{n\to\infty}\ma_a^S(\eta_n,\beta_0)\leq0 \, .
    \end{equation*}
    Hence, all inequalities are actually equalities, and in particular, $\ma_a^S(\eta_n,\beta_0)\to0$ as $n\to\infty$. Altogether, we obtain 
    \begin{equation*}
        \begin{split}
            \sigma\cosh(\ma_a^S(\gamma,\beta_0))&\geq\limsup_{n\to\infty}\sigma\cosh(\ma_a^S(\gamma,\eta_n)+\ma_a^S(\eta_n,\beta_0))\\
            &=\limsup_{n\to\infty}\sigma\cosh(\ma_a^S(\gamma,\eta_n))\\
            &\geq\limsup_{n\to\infty}\sigma\cosh(\widetilde\ma_a^{S,K}(b_n,c_n)) \, ,
        \end{split}
    \end{equation*}
    where the last inequality follows from angle monotonicity in $(\leq K)$-curvature comparison neighbourhoods. Combining this with \eqref{eq:lhs} finishes this part of the proof.

    Finally, let $U$ be a $(\geq K)$-comparison neighbourhood. Consider the triangles $\triangle (a,b_n, c_n)$. We know that $\tau_s(a,c_n)\to 0$ and 
    \[\lim_{n\to\infty}\tau_s(b_n,a)=\lim_{n\to\infty}\tau_s(b_n,c_n)=\tau_s(b,a)\, ,\]
    so by the Lorentzian law of cosines, we must have $\widetilde\ma_{b_n}^{S,K}(a,c_n)\to 0$ as $n\to\infty$. Now note that the area of the comparison triangles $\triangle(\tilde a,\tilde b_n, \tilde c_n)$ in $\mathbb L^2(K)$ goes to zero too, hence the Gauss-Bonnet formula \Cref{gbtt} for timelike triangles yields 
    \[\limsup_{n\to\infty}\widetilde\ma_a^{S,K}(b_n,c_n)=-\liminf_{n\to\infty}\widetilde\ma_{c_n}^{S,K}(a,b_n) \, .\]
    Using this fact, as well as angle monotonicity, gives
    \begin{align*}
        \limsup_{n\to\infty}\sigma\cosh(\widetilde\ma_a^{S,K}(b_n,c_n))&=\limsup_{n\to\infty}\sigma\cosh(\widetilde\ma_{c_n}^{S,K}(a,b_n))\\\
        &\leq\limsup_{n\to\infty}\sigma\cosh(\ma^S_{c_n}(a,b_n))\\
        &=\limsup_{n\to\infty}\sigma\cosh(\ma^S_{c_n}(b_n,\gamma(\delta)))\\
        &=\sigma\cosh\left(\limsup_{n\to\infty}\ma^S_{c_n}(b_n,\gamma(\delta))\right)\\\
        &\leq\sigma\cosh(\ma_a^S(b,\gamma(\delta)))=\sigma\cosh(\ma_a(\beta_0,\gamma)) \, ,
    \end{align*}
    where \Cref{rem:CBBfact}\ref{CBBfact2} was used to go from the second to the third line and semi-continuity of angles to go from the fourth to the fifth line. Combining this chain of inequalities with \eqref{eq:lhs} finishes the proof.
\end{proof}

\begin{rem}[On assumptions in the first variation formula]\label{rem:assumptions on fvf}
    The previous theorem only holds for strongly causal, locally causally closed and regular Lorentzian pre-length spaces. For curvature bounded from above by $K$, one could drop these assumptions and instead assume $U=X$, i.e.\ the entire space is a $(\leq K)$-curvature comparison neighbourhood. In this case, one can apply \Cref{lem:firstvargeq} directly to (comparison triangles of) 
    the triangles $\triangle (p,a,c_n)$. This yields 
    \[\sigma\cosh(\ma_a^S(\gamma,\beta_0))\geq\sigma\cosh(\widetilde\ma_a^{S,K}(p,c_n))\geq\frac{l(t_n)-l(0)}{t_n} \, .\]
    Note that this argument works for any choice of $t_n$ and does not need to assume the uniform convergence of the curves $\beta_{t_n}$. In particular, combining this with \Cref{prop:geq}, the first variation formula can be written as
    \[\lim_{t\searrow0}\frac{l(t)-l(0)}{t}=\sigma\cosh(\ma_a(\gamma,\beta_0)) \, .\]
\end{rem}

\section{Rigidity for upper curvature bounds}
In this section, we prove a rigidity result for the case of equality in the triangle comparison and angle comparison conditions for spaces with timelike curvature bounded above, analogous to \cite[Proposition~9.1.19]{burago-burago-ivanov2001}. In particular, we obtain a criterion to detect flat regions in a Lorentzian pre-length space with timelike curvature bounded above by $0$ (\Cref{pop:quadrangle rigidity}), analogous to \cite[Lemma~9.2.30]{burago-burago-ivanov2001}. This criterion will be a key ingredient in the proof of \Cref{thm:main}, since it allows us to prove that, under our curvature assumption, weakly parallel lines are actually synchronised-parallel lines.

For the following results, recall that for any $S\subseteq \mathbb{R}^n$, the \textit{convex hull} of $S$ is the set

\[
\conv(S) := \left\{\sum_{i=0}^N a_ix_i:N\in\mathbb N,\ x_i\in S,\ a_i\geq 0\ \text{with}\ \sum_{i=0}^N a_i=1\right\}.
\]

\begin{pop}[Rigidity of upper curvature bounds]
\label{pop:triangleCBequality}
Let $X$ be a \LpLSn.
Let $\triangle (a,b,c)$ be a timelike triangle with a certain order of the vertices in a causally closed ($\leq K$)-curvature comparison neighbourhood $U \subseteq X$. 
Let $\triangle (\bar a, \bar b, \bar c)$ be a comparison triangle in $\lm{K}$.
Then the following are equivalent:
\begin{enumerate}[label=(\roman*)]
\item The angles at $a$ and $\bar a$ are equal, i.e.\ $\ma_a(b,c) =\widetilde{\ma}_a^K(b,c)$.
\item All the angles of $\triangle (a,b,c)$ are equal to their corresponding angles in $\triangle (\bar a, \bar b, \bar c)$. 
\item For all points $x\in[a,c]$, $x\neq a,c$, its comparison point $\bar x$ satisfies $\tau_s(b,x)=\tau_s(\bar b,\bar x)$.
\item There exists a point $x\in[a,c]$, $x\neq a,c$ for which its comparison point $\bar x$ satisfies $\tau_s(b,x)=\tau_s(\bar b,\bar x)>0$.
\item There is an isometry $f$ from the `filled-in triangle' $\conv(\triangle(\bar a,\bar b,\bar c))$ into $U$ mapping the sides of $\triangle(\bar a,\bar b,\bar c)$ onto the corresponding sides of $\triangle(a,b,c)$ (`there is a fill-in of $\triangle(a,b,c)$ which is a part of the model space'). The image of $f$ is geodesically convex, i.e.\ for any $x\leq y$ in the image of $f$ there exists a geodesic connecting them that stays inside.
\end{enumerate}
\end{pop}

\begin{figure}[h]
\centering
\begin{tikzpicture}[scale=3, thick]

\coordinate (a) at (0,0);
\coordinate (b) at (0.4,0.5);
\coordinate (c) at (0,1.4);
\coordinate (x) at (0,1);

\draw (a) -- (b) -- (c) -- cycle;
\draw (b) -- (x);

\node[below]  at (a) {$\ba$};
\node[right]  at (b) {$\bb$};
\node[above]  at (c) {$\bc$};
\node[left]  at (x) {$\bx$};

\coordinate (mAB) at ($ (a)!0.5!(b) $);
\coordinate (mBC) at ($ (b)!0.5!(c) $);
\coordinate (mAC) at ($ (a)!0.5!(c) $);

\node[below right] at ($ (mAB) $) {$\bar\gamma$};
\node[above right]  at ($ (mBC)   $) {$\bar \alpha$};
\node[left]  at ($ (mAC) $) {$\bar\beta$};
\fill (a) circle (0.6pt);
\fill (b) circle (0.6pt);
\fill (c) circle (0.6pt);
\fill (x) circle (0.6pt);

\pic[draw, angle radius=30.pt, angle eccentricity=1.2]
       {angle = b--a--c};
\node at ($(a)+(0.07,0.2)$) {$\theta$};
\end{tikzpicture}
\caption{Labelling of the comparison configuration.}
\label{fig: rigidity notation}
\end{figure}
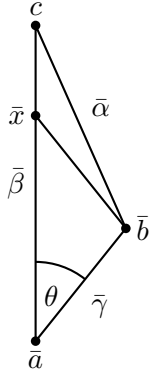

\begin{proof} 
We show $(i) \Rightarrow (iii) \Rightarrow (ii) \Rightarrow (i)$, 
proving they are equivalent. 
Finally, we show that $(iv)\Leftrightarrow (iii)$ and $(i) + (ii) + (iii) + (iv) \Leftrightarrow (v)$.

\fbox{$(i) \Rightarrow (iii)$} Let $\alpha = [b,c]$, $\beta=[a,c]$ and $\gamma=[a,b]$, and let $\bar\alpha=[\bar b,\bar c]$, $\bar\beta=[\bar a,\bar c]$ and $\bar \gamma = [\bar a, \bar b]$ be defined the same way in the model space $\lm{K}$. Moreover, all of these curves are parametrized with constant speed and defined on $[0,1]$.

Let $\theta(s,t):=\widetilde{\ma}_a^K(\beta(s),\gamma(t))$. 
Since $X$ has timelike curvature bounded above, $\theta(s,t)$ is monotonically decreasing, cf.\ \cite[Definition 3.8]{beran-kunzinger-rott2024}. 
If $(i)$ is true, then we have that $\theta(1,1)=\lim_{s,t\to0}\theta(s,t)$, and so $\theta(s,t)$ must be constant where it is defined. 
By the Lorentzian law of cosines, the equality $\theta(s,t)=\theta
(1,1)$ means that as two of the side lengths between $\triangle(a,\beta(s),\gamma(t))$ and $\triangle(\bar a,\bar\beta(s),\bar\gamma(t))$ agree and the comparison angle at $a$ is equal to the angle at $\bar a$ in the triangle in the model space, also the third side length agrees. That is, $\tau_s(\beta(s),\gamma(t))=\tau_s(\bar\beta(s),\bar\gamma(t))$ whenever this is positive. If $\tau_s(\beta(s),\gamma(t))=0$, then this equality is trivially satisfied by standard triangle comparison. Now we specify $t=1$ making $\gamma(1)=b$, yielding (iii). 
Note that instead of putting $t=1$ we can put $s=1$, yielding (iii) with $b$ and $c$ interchanged.

\fbox{$(iii)\Rightarrow(ii)$} 
Assume $(iii)$, that is, for all points $\beta(t)$ with $t \in (0,1)$, its comparison point $\bar \beta(t)$ satisfies $\tau_s(b, \beta(t)) = \bar \tau_s(\bar b, \bar \beta(t))$ for all $t \in (0,1)$. Let $l(t) = \tau_s(b, \beta(t))$ and $\bar l(t) = \bar \tau_s(\bar b, \bar \beta(t))$. Then clearly $l(t) = \bar l(t)$ and applying the first variation formula (\Cref{fvf}) twice for both $t \to 0$ and $t \to 1$, yields $\ma_a (b,c) = \widetilde \ma_a^K (b,c)$ and $\ma_c (a,b) = \widetilde \ma_c^K (a,b)$. It remains to show that $\ma_b (a,c) = \widetilde \ma_b^K (a,c)$. To do this, we simply apply $(i) \Rightarrow (iii)$ to the angle $\ma_a (b,c)$. This shows condition $(iii)$ holds for the side $\alpha$, now applying the argument above yields the final angle.

\fbox{$(ii) \Rightarrow (i)$} There is nothing to show, as $(i)$ is a weaker statement.

\fbox{$(iii) \Rightarrow (iv)$} There is nothing to show, as $(iv)$ is a weaker statement.

\fbox{$(iv) \Rightarrow (iii)$} If $(iv)$ is satisfied at $x\in\beta$, by time symmetry, without loss of generality, we can assume $a\ll x\ll c$ and $x\ll b$, see \Cref{fig: rigidity notation}. 
$(iv)$ implies that the comparison point $\bar x$ for $x$ forms valid comparison triangles $\triangle(\bar a,\bar b,\bar x)$ for $\triangle(a,b,x)$ and $\triangle(\bar b,\bar x,\bar c)$ for $\triangle(b,x,c)$. By angle additivity in the plane, we have 
\[
\widetilde\ma^K_{x}(a,b)=\ma_{\bar x}(\bar a,\bar b)=\ma_{\bar x}(\bar b,\bar c)=\widetilde\ma^K_{x}(b,c)
\] 
and by timelike curvature bounded above, we have 
\[
\ma_{x}(a,b)\geq\widetilde\ma^K_{x}(a,b)=\widetilde\ma^K_{x}(b,c)\geq\ma_{x}(b,c) \, .
\]
By \Cref{pop:triangle inequality angles} \ref{FuFuPa/PaPaFu}, we have that 
\[
\ma_x(a,b)\leq\ma_x(a,c)+\ma_x(c,b)=\ma_x(c,b) \, ,
\]
thus they are equal, so $(i)$ holds for $\triangle(a,b,x)$ and $\triangle(b,x,c)$ at $x$. We use that $(i) \Rightarrow (iii)$ on these two triangles and get $(iii)$ for the original triangle. Namely, let $y\in\beta$, then it is either in $[a,x]$ or in $[x,c]$, say in $[a,x]$. As $\triangle(\bar a,\bar b,\bar x)$ is the comparison triangle for $\triangle(a,b,x)$ and a comparison point for $y$ is also a comparison point for $y$ in $\triangle(\bar a,\bar b,\bar c)$, we have that $\tau(b,x)=\tau(\bar b,\bar x)$ and $\tau(x,b)=\tau(\bar x,\bar b)$.

\fbox{$(v)\Rightarrow(i)$---$(iv)$} This is clear by the definition of an isometry.

\fbox{$(i)$---$(iv)\Rightarrow(v)$} Assume that $a\ll b\ll c$. Note that all points in $\conv(\triangle(\bar a,\bar b,\bar c))$ are of the form $[\bar a,[\bar b,\bar c](s)](t)$ for some $s,t\in[0,1]$, and $(s,t)$ are unique unless $t=0$. Define a map 
\begin{align*}
f\colon \conv(\triangle(\bar a,\bar b,\bar c)) &\longrightarrow X \, ,\\
[\bar a,[\bar b,\bar c](s)](t) & \longmapsto [a,[ b, c](s)](t) \, .
\end{align*}
This is well-defined since for $t=0$, this always yields $a$, and because geodesics in comparison neighbourhoods are unique by \cite[Theorem 4.7]{BNR25}.
We claim that, if $x=[b,c](s_1)\ll y=[b,c](s_2)$, then $(i)$ also holds at $a$ in $\triangle(a,x,y)$.
We know $\bar x=[\bar b,\bar c](s_1),\bar y=[\bar b,\bar c](s_2)$.
By \Cref{pop:triangle inequality angles} \ref{FuFuFu/PaPaPa}, we have that
\begin{align*}
\widetilde\ma^K_a(b,c)&=\widetilde\ma^K_a(b,x)+\widetilde\ma^K_a(x,y)+\widetilde\ma^K_a(y,c)\\
&\geq \ma_a(b,x)+\ma_a(x,y)+\ma_a(y,c)\geq \ma_a(b,c)=\widetilde\ma^K_a(b,c) \, .
\end{align*}
Therefore, all of these inequalities are equalities. 
Inspecting this more closely, we get $\widetilde\ma^K_a(x,y)=\ma_a(x,y)$, i.e.\ $(i)$ holds for $\triangle(a,x,y)$, and the subtriangle $\triangle(\bar a,\bar x,\bar y)$ is its comparison triangle.
Now for two points $w=[a,x](t_1),z=[a,y](t_2)$ we know that the comparison points are $\bar w=[\bar a,\bar x](t_1),\bar z=[\bar a,\bar y](t_2)$ and $f(\bar w)=w$ and $f(\bar z)=z$. 
By an analogous argument to that of $(i) \Rightarrow (iii)$, one can show that $\tau(w,z)=\tau(\bar w,\bar z)$. Varying $s_1,t_1$ and $s_2,t_2$ gives all comparison points, and it follows that $f$ is $\tau$-preserving.

To prove that $f$ is $\leq$-preserving, observe that the implication $f(z)\leq f(w)\Rightarrow z\leq w$ follows easily by approximating $z$ or $w$ in a timelike way: Say we find $z_n \to z$ with $z_n \ll z$, then $f(z_n)\ll f(z)\leq f(w)$, thus $z_n\ll w$. In the limit we still have $z\leq w$. 
The implication $z\leq w\Rightarrow f(z)\leq f(w)$ follows similarly using that $U$ is causally closed.
\end{proof}

\begin{pop}[Quadrangle with signed angle sum zero]\label{pop:quadrangle rigidity}
Let $X$ be a globally causally closed \LpLS with timelike curvature globally bounded above by $K=0$. Consider four points $p_1\ll p_2\ll p_4\ll p_3$ as a quadrangle (with timelike diagonals) $Q=(p_1,p_2,p_3,p_4)$. If the angles of this quadrangle satisfy
\begin{equation}\label{eq:quadranglerigidity}
\ma_{p_1}(p_2,p_4) + \ma_{p_3}(p_2,p_4) \geq \ma_{p_2}(p_1,p_3) + \ma_{p_4}(p_1,p_3) \,,
\end{equation}
then this is an equality and $Q$ bounds a convex flat region. More precisely, there is a filled-in quadrangle $\bar Q=\conv(\bar p_1,\bar p_2,\bar p_3,\bar p_4)\subseteq\R^{1,1}$ in Minkowski space and an isometry $f:\bar Q\to X$ mapping $\bar p_i$ to $p_i$ for all $i$.
\end{pop}
\begin{proof} 
Recall that since we are working with $K=0$, we drop the $K$ from the comparison angles.
Take the triangles $\triangle_1:=\triangle(p_1,p_2,p_4),\triangle_3:=\triangle(p_2,p_3,p_4)$, and
consider a comparison configuration $\bar Q$ consisting of comparison triangles $\bar\triangle_1,\bar\triangle_3$ of $\triangle_1,\triangle_3$ with a shared side $[\bar p_2,\bar p_4]$, with $\bar p_1,\bar p_3$ on opposite sides of the line extending the side $[\bar p_2,\bar p_4]$, see \Cref{fig: quadrangle 1}. 
\begin{figure}
\begin{center}
\begin{tikzpicture}[scale=0.5, line cap=round, line join=round, >=triangle 45]

\begin{scope}
  \clip(-3,-1) rectangle (4,10);

  \fill[orange!50, opacity=0.7] (0,0)    .. controls (-1.5,2) ..     (-2,4)  .. controls (0.8,5.75) ..  (3,7.5)  .. controls (2,3.75) ..  cycle; 
  \fill[blue!50, opacity=0.7] (-2,4)  .. controls (-1,6.5) .. (1,9) .. controls (2.2,8.25) .. (3,7.5) .. controls (0.8,5.75) .. cycle; 
  
  \draw (0,0)  .. controls (-1.5,2) ..  (-2,4);
  \draw (-2,4) .. controls (-1,6.5) .. (1,9);
  \draw (1,9) .. controls (2.2,8.25) .. (3,7.5);
  \draw (-2,4) .. controls (0.8,5.75) .. (3,7.5);
  \draw (0,0) .. controls (2,3.75) .. (3,7.5);
  \draw (1,9) .. controls (0.7,4.5) .. (0,0);

  \fill (0,0) circle (3pt);   \node[below left]  at (0,0)   {$p_1$};
  \fill (-2,4) circle (3pt);  \node[left]        at (-2,4)  {$p_2$};
  \fill (1,9) circle (3pt);   \node[above]       at (1,9)   {$p_3$};
  \fill (3,7.5) circle (3pt); \node[right]       at (3,7.5) {$p_4$};

  \node[text=orange!80!black] at (-0.7,3.5) {$\triangle_1$};
  \node[text=blue!80!black] at (1.6,7.3)  {$\triangle_3$};
\end{scope}

\begin{scope}[shift={(10,0)}]
  \clip(-3,-1) rectangle (4,10);

  \fill[orange!50, opacity=0.7] (0,0) -- (-2,4) -- (3,7.5) -- cycle; 
  \fill[blue!50, opacity=0.7] (-2,4) -- (1,9) -- (3,7.5) -- cycle; 

  \draw (0,0)-- (-2,4);
  \draw (-2,4)-- (1,9);
  \draw (1,9)-- (3,7.5);
  \draw (-2,4)-- (3,7.5);
  \draw (0,0)-- (3,7.5);
  \draw (1,9)-- (0,0);

  \fill (0,0) circle (3pt);   \node[below left]  at (0,0)   {$\bar{p}_1$};
  \fill (-2,4) circle (3pt);  \node[left]        at (-2,4)  {$\bar{p}_2$};
  \fill (1,9) circle (3pt);   \node[above]       at (1,9)   {$\bar{p}_3$};
  \fill (3,7.5) circle (3pt); \node[right]       at (3,7.5) {$\bar{p}_4$};

  \node[text=orange!80!black] at (-0.7,3.5) {$\bar\triangle_1$};
  \node[text=blue!80!black] at (1.6,7.3)  {$\bar\triangle_3$};
\end{scope}

\end{tikzpicture}
\end{center}
\caption{Setup of the quadrangles, first case.}
    \label{fig: quadrangle 1}
\end{figure}
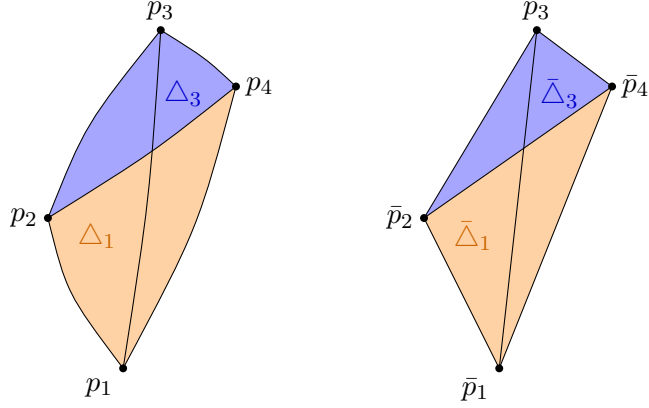

Since we are comparing with the Minkowski plane, by \eqref{eq:angle condition} we have that the angles in the triangles $\triangle_1$ and $\triangle_3$ satisfy the following inequalities: 
\begin{align}\label{eq:quadranglerigidity:angleComparison}
\begin{split}
\ma_{p_1}(p_2,p_4)\leq\widetilde\ma_{p_1}(p_2,p_4)\, ,\qquad\ma_{p_2}(p_3,p_4)\leq\widetilde\ma_{p_2}(p_3,p_4)\, ,\\
\ma_{p_2}(p_1,p_4)\geq\widetilde\ma_{p_2}(p_1,p_4)\, ,\qquad\ma_{p_3}(p_2,p_4)\leq\widetilde\ma_{p_3}(p_2,p_4)\, ,\\
\ma_{p_4}(p_1,p_2)\leq\widetilde\ma_{p_4}(p_1,p_2)\, ,\qquad\ma_{p_4}(p_2,p_3)\geq\widetilde\ma_{p_4}(p_2,p_3)\, .
\end{split}
\end{align}
By \Cref{lem:angleSumMink}, we have that 
\begin{align*}
\widetilde\ma_{p_1}(p_2,p_4)+\widetilde\ma_{p_4}(p_1,p_2)=\widetilde\ma_{p_2}(p_1,p_4) \, ,\\
\widetilde\ma_{p_2}(p_3,p_4)+\widetilde\ma_{p_3}(p_2,p_4)=\widetilde\ma_{p_4}(p_2,p_3) \, ,
\end{align*}
and in particular
\begin{align}\label{eq:quadranglerigidity:maOverTriangle}
\begin{split}
\ma_{p_1}(p_2,p_4)+\ma_{p_4}(p_1,p_2)\leq\ma_{p_2}(p_1,p_4)\, ,\\
\ma_{p_2}(p_3,p_4)+\ma_{p_3}(p_2,p_4)\leq\ma_{p_4}(p_2,p_3)\, .
\end{split}
\end{align}
By \Cref{pop:triangle inequality angles} \ref{FuFuPa/PaPaFu} at $p_2$ and $p_4$,
we have that:
\begin{align}\label{eq:quadranglerigidity:triIneqAng}
\begin{split}
\ma_{p_2}(p_1,p_4)\leq\ma_{p_2}(p_3,p_4)+\ma_{p_2}(p_1,p_3)\, ,\\
\ma_{p_4}(p_2,p_3)\leq\ma_{p_4}(p_2,p_1)+\ma_{p_4}(p_1,p_3) \, .
\end{split}
\end{align}
Notice that all but the last term appears in the sum of \eqref{eq:quadranglerigidity:maOverTriangle} and thus we can conclude that
\begin{align*}
\ma_{p_1}(p_2,p_4)\!+\!\ma_{p_3}(p_2,p_4)&\leq \ma_{p_2}(p_1,p_4)\!-\!\ma_{p_2}(p_3,p_4)\!+\!\ma_{p_4}(p_2,p_3)\!-\!\ma_{p_4}(p_1,p_2)\\
&\leq \ma_{p_2}(p_1,p_3)\!+\!\ma_{p_4}(p_1,p_3) \, ,
\end{align*}
which is the converse inequality of the assumption. Thus, all our inequalities are, in fact, equalities. In particular, all the angles in $\triangle_1$ and $\triangle_3$ are equal to the respective comparison angles.

We claim that $\bar{Q}$ is convex. Indeed, this boils down to $\bar{p}_2$ and $\bar{p}_4$ being on opposite sides of the segment $[\bar{p}_1, \bar{p}_3]$. 
By \eqref{eq:quadranglerigidity} and the angle comparison condition given by the curvature assumption, we have
\begin{equation}\label{eq:quadranglerigidity:comparisonangles}
\widetilde{\ma}_{p_1}(p_2,p_4) + \widetilde{\ma}_{p_3}(p_2,p_4) \geq \widetilde{\ma}_{p_2}(p_1,p_3) + \widetilde{\ma}_{p_4}(p_1,p_3) \, .
\end{equation}
However, if $\bar{p}_2$ and $\bar{p}_4$ are on the same side of the line extending $\bar{p}_1\bar{p}_3$, then, by our initial assumption of $\bp_1$ and $\bp_3$ being on opposite sides of the line extending the segment $[\bp_2,\bp_4]$, we get that either $\bar p_2$ is in the interior of $\triangle(\bar{p}_1,\bar{p}_3,\bar{p}_4)$ or $\bar{p}_4$ is in the interior of $\triangle(\bar{p}_1,\bar{p}_2,\bar{p}_3)$. Without loss of generality, we assume the latter.

Under this hypothesis, it is clear that 
\[
\widetilde{\ma}_{p_1}(p_2,p_4) \leq \widetilde{\ma}_{p_1}(p_2,p_3) \quad \text{and} \quad \widetilde{\ma}_{p_3}(p_2,p_4) \leq \widetilde{\ma}_{p_3}(p_2,p_1) 
\]
which implies
\[
\widetilde{\ma}_{p_1}(p_2,p_4) + \widetilde{\ma}_{p_3}(p_2,p_4)  \leq \widetilde{\ma}_{p_1}(p_2,p_3) + \widetilde{\ma}_{p_3}(p_2,p_1)  = \widetilde{\ma}_{p_2}(p_1,p_3)\, 
\]
This, combined with \eqref{eq:quadranglerigidity:comparisonangles} yields 
$\widetilde{\ma}_{p_4}(p_1,p_3) = 0$, which contradicts that $\bar{p}_4$ is in the interior of $\triangle(\bar{p}_1,\bar{p}_2,\bar{p}_3)$.

Let $\bar q$ be the intersection of the diagonals $[\bar p_1,\bar p_3]\cap[\bar p_2,\bar p_4]$ and $q=[p_2,p_4](\tau(\bar p_2,\bar q))$ (i.e.\ $\bar q$ is its comparison point, both for $\triangle_1$ and $\triangle_3$). It exists by convexity of $\bar Q$.
Now we have the following chain of equalities:
\begin{align*}
\ma_{p_2}(p_1,p_3)&=\ma_{p_2}(p_1,p_4)-\ma_{p_2}(p_3,p_4)\\
&=\widetilde\ma_{p_2}(p_1,p_4)-\widetilde\ma_{p_2}(p_3,p_4)\\
&=\ma_{\bar p_2}(\bar p_1,\bar p_3)
\end{align*}
which hold by equality in \eqref{eq:quadranglerigidity:triIneqAng}, equality in \eqref{eq:quadranglerigidity:angleComparison} and angle additivity in the plane. Thus, $\bar p_1,\bar p_2,\bar p_3$ forms a valid comparison hinge for $[p_2,p_1],[p_2,p_3]$. By hinge comparison and reverse triangle inequality, this yields 
\[
\tau(p_1,p_3)\leq\tau(\bar p_1,\bar p_3)=\tau(\bar p_1,\bar q)+\tau(\bar q,\bar p_3)\leq\tau(p_1,p_3)\,.
\]
Therefore, $\bar\triangle_2=\triangle(\bar p_1,\bar p_2,\bar p_3)$ and $\bar\triangle_4=\triangle(\bar p_1,\bar p_3,\bar p_4)$ are comparison triangles for $\triangle_2:=\triangle(p_1,p_2,p_3)$, $\triangle_4:=\triangle(p_1,p_3,p_4)$, see \Cref{fig: quadrangle 2}. 

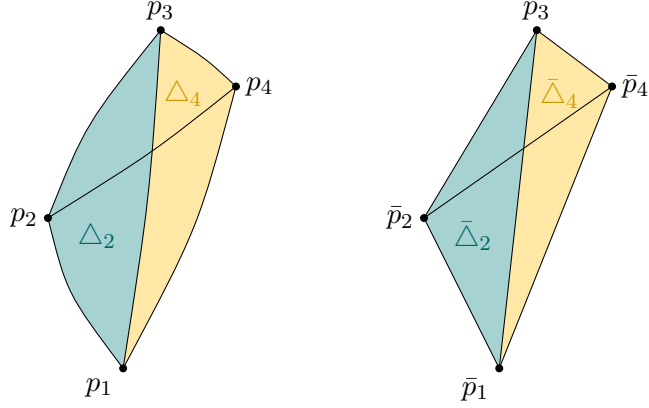
\begin{figure}
\begin{center}
\begin{tikzpicture}[scale=0.5, line cap=round, line join=round, >=triangle 45]

\begin{scope}
  \clip(-3,-1) rectangle (4,10);

  \fill[teal!50, opacity=0.7] (0,0) .. controls (-1.5,2) .. (-2,4) .. controls (-1,6.5) .. (1,9) .. controls (0.7,4.5) .. cycle; 
  \fill[Amber!50, opacity=0.7] (0,0) .. controls (0.7,4.5) .. (1,9) .. controls (2.2,8.25) .. (3,7.5) .. controls (2,3.75) .. cycle;

  \draw (0,0)  .. controls (-1.5,2) ..  (-2,4);
  \draw (-2,4) .. controls (-1,6.5) .. (1,9);
  \draw (1,9) .. controls (2.2,8.25) .. (3,7.5);
  \draw (-2,4) .. controls (0.8,5.75) .. (3,7.5);
  \draw (0,0) .. controls (2,3.75) .. (3,7.5);
  \draw (1,9) .. controls (0.7,4.5) .. (0,0);

  \fill (0,0) circle (3pt);   \node[below left]  at (0,0)   {$p_1$};
  \fill (-2,4) circle (3pt);  \node[left]        at (-2,4)  {$p_2$};
  \fill (1,9) circle (3pt);   \node[above]       at (1,9)   {$p_3$};
  \fill (3,7.5) circle (3pt); \node[right]       at (3,7.5) {$p_4$};

  \node[text=teal!80!black] at (-0.7,3.5) {$\triangle_2$};
  \node[text=Amber!80!black] at (1.6,7.3)  {$\triangle_4$};
\end{scope}

\begin{scope}[shift={(10,0)}] 
  \clip(-3,-1) rectangle (4,10);

  \fill[teal!50, opacity=0.7] (0,0) -- (-2,4) -- (1,9) -- cycle;
  \fill[Amber!50, opacity=0.7] (0,0) -- (1,9) -- (3,7.5) -- cycle; 
  
  \draw (0,0)-- (-2,4);
  \draw (-2,4)-- (1,9);
  \draw (1,9)-- (3,7.5);
  \draw (-2,4)-- (3,7.5);
  \draw (0,0)-- (3,7.5);
  \draw (1,9)-- (0,0);

  \fill (0,0) circle (3pt);   \node[below left]  at (0,0)   {$\bar{p}_1$};
  \fill (-2,4) circle (3pt);  \node[left]        at (-2,4)  {$\bar{p}_2$};
  \fill (1,9) circle (3pt);   \node[above]       at (1,9)   {$\bar{p}_3$};
  \fill (3,7.5) circle (3pt); \node[right]       at (3,7.5) {$\bar{p}_4$};

  \node[text=teal!80!black] at (-0.7,3.5) {$\bar\triangle_2$};
  \node[text=Amber!80!black] at (1.6,7.3)  {$\bar\triangle_4$};
\end{scope}
\end{tikzpicture}
\end{center}
    \caption{Setup of the quandrangles, second case.}
    \label{fig: quadrangle 2}
\end{figure}

An argument analogous to the above can be formulated for the triangles $\triangle_2,\triangle_4$: 
Since we are comparing to a $K=0$ comparison space, by \eqref{eq:angle condition} we have that  
\begin{align}\label{eq:quadranglerigidity:Delta24Chain}
\begin{split}
&\widetilde\ma_{p_2}(p_1,p_3)+\widetilde\ma_{p_4}(p_1,p_3)\\
&\leq\ma_{p_2}(p_1,p_3)+\ma_{p_4}(p_1,p_3)\\
&\leq\ma_{p_1}(p_2,p_4)+\ma_{p_3}(p_2,p_4)\\
&\leq \ma_{p_1}(p_2,p_3)+\ma_{p_1}(p_3,p_4) + \ma_{p_3}(p_2,p_1)+\ma_{p_3}(p_1,p_4)\\
&\leq \widetilde\ma_{p_1}(p_2,p_3)+\widetilde\ma_{p_1}(p_3,p_4) + \widetilde\ma_{p_3}(p_2,p_1)+\widetilde\ma_{p_3}(p_1,p_4) \, ,
\end{split}
\end{align}
where we have subsequently applied the assumption; \Cref{pop:triangle inequality angles} \ref{FuFuFu/PaPaPa}; and \eqref{eq:angle condition} again.
Additionally, by \Cref{lem:angleSumMink} we know that
\begin{align*}
\widetilde\ma_{p_2}(p_1,p_3)=\widetilde\ma_{p_1}(p_2,p_3)+\widetilde\ma_{p_3}(p_1,p_2)\, ,\\
\widetilde\ma_{p_4}(p_1,p_3)=\widetilde\ma_{p_1}(p_3,p_4)+\widetilde\ma_{p_3}(p_1,p_4) \, .
\end{align*}
Thus, the start and end of \eqref{eq:quadranglerigidity:Delta24Chain} are equal, and in particular, all angles in $\triangle_2$ and $\triangle_4$ are equal to the respective comparison angles.

Thus, we know that \Cref{pop:triangleCBequality}(i) holds for all $\triangle_i$ and thus also \Cref{pop:triangleCBequality}(v). 
That is, we get isometries $f_i$ from the filled-in triangles $\bar\triangle_i$ into $\triangle_i$ providing a fill-in of $\triangle_i$. As $\bar\triangle_1,\bar\triangle_3$ only overlap on the side $[\bar p_2,\bar p_4]$ where $f_1$ and $f_3$ agree, these isometries fit together to a continuous map $f_{13}$ from the filled-in quadrangle $\bar Q$ to $X$. The same works with $f_2$ and $f_4$, giving a continuous map $f_{24}$.

Both of these maps map the sides of the quadrangle $[\bar p_i,\bar p_{i+1}]$ to $[p_i,p_{i+1}]$ (taking indices modulo 4),  
$f_{13}$ maps the diagonal $[\bar p_2,\bar p_4]$ to $[p_2,p_4]$, and $f_{24}$ maps the diagonal $[\bar p_1,\bar p_3]$ to $[p_1,p_3]$, as each of these lie in an isometry region of the map.

For the other diagonals, we already know
\[
\tau(p_1,p_3)=\tau(\bar b_1,\bar q)+\tau(\bar q,\bar p_3)=L(f_{13}([\bar p_1,\bar p_3]))\,,
\]
making $f_{13}([\bar p_1,\bar p_3])$ equal to $[p_1,p_3]$. 
For $f_{24}([\bar p_2,\bar p_4])$, one can proceed analogously: set $q'=f_{24}(\bar q)$, prove that $\tau(p_2,p_4)=\tau(p_2,q')+\tau(q',p_4)$ with hinge comparison, making $f_{24}([\bar p_2,\bar p_4])=[p_2,p_4]$. 
Now note that both $q,q'$ are the intersection of $[p_1,p_3]$ and $[p_2,p_4]$ and are thus the same (if the intersection is not just a point, note $\tau(p_1,q)=\tau(p_1,q')$).

Now we define the four subtriangles $\bar\triangle_{12}=\triangle(\bar p_1,\bar p_2,\bar q)$, $\bar\triangle_{23}=\triangle(\bar p_2,\bar p_3,\bar q)$, $\bar\triangle_{34}=\triangle(\bar p_3,\bar p_4,\bar q)$, $\bar\triangle_{41}=\triangle(\bar p_4,\bar p_1,\bar q)$ (where the filled-in triangle $\bar\triangle_{ij}$ is the intersection of the filled-in triangles $\bar\triangle_i$ and $\bar\triangle_j$), see \Cref{fig: planar quadrangle four subtriangles}. We claim that $f_{13}=f_{24}$.  
Indeed, we already know that all $[\bar p_i,\bar q]$ are mapped isometrically to $[p_i,q]$, i.e.\ $f_{13},f_{24}$ agree on the diagonals. Note now that any point $\bar x$ in the filled-in quadrangle $\bar Q$ is in one of the filled-in triangles $\bar\triangle_{ij}$. Let $\bar a\ll \bar b\ll \bar c$ be the vertices of $\bar\triangle_{ij}$ in time order, then $\bar x$ can be described as $\bar x=[\bar a,[\bar b,\bar c](s)](t)$ for some $s,t$. As both $f_{13}$ and $f_{24}$ are isometries, both $f_{13}([\bar a,[\bar b,\bar c](s)])$ and $f_{24}([\bar a,[\bar b,\bar c](s)])$ are geodesics connecting $f_{13}(\bar a)=f_{24}(\bar a)$ to $f_{13}([\bar b,\bar c](s))=f_{24}([\bar b,\bar c](s))$ and are thus the same by uniqueness of geodesics in comparison neighbourhoods. In particular, $f_{13}(\bar x)=f_{24}(\bar x)$.

Finally, we claim that $f:=f_{13}=f_{24}$ is an isometry. In order to prove this, let 
$\bar x,\bar y$ be in the filled-in quadrilateral $\bar Q$. They are in some subtriangles $\bar\triangle_{ij}$ and $\bar\triangle_{kl}$. We consider four cases.

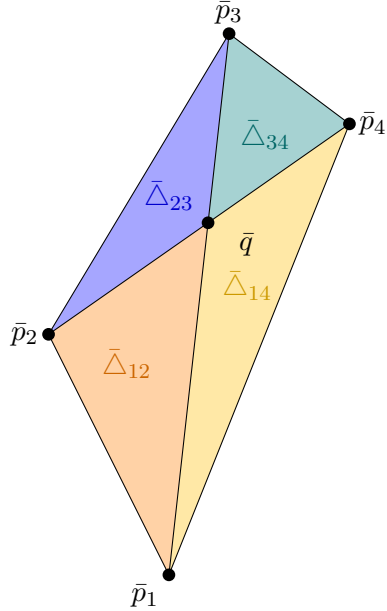
\begin{figure}
\begin{center}
\begin{tikzpicture}[scale=0.8, line cap=round, line join=round, >=triangle 45]
\begin{scope}
  \clip(-3,-1) rectangle (4,10);

  \fill[orange!50, opacity=0.7] (0,0) -- (-2,4) -- (0.6506,5.8554) -- cycle; 
  \fill[blue!50, opacity=0.7] (-2,4) -- (1,9) -- (0.6506,5.8554) -- cycle; 
  \fill[teal!50, opacity=0.7] (1,9) -- (3,7.5) -- (0.6506,5.8554) -- cycle; 
  \fill[Amber!50, opacity=0.7] (0,0) -- (3,7.5) -- (0.6506,5.8554) -- cycle; 

  \draw (0,0)-- (-2,4);
  \draw (-2,4)-- (1,9);
  \draw (1,9)-- (3,7.5);
  \draw (-2,4)-- (3,7.5);
  \draw (0,0)-- (3,7.5);
  \draw (1,9)-- (0,0);

  \fill (0,0) circle (3pt);   \node[below left]     at (0,0)   {$\bar p_1$};
  \fill (-2,4) circle (3pt);  \node[left]           at (-2,4)  {$\bar p_2$};
  \fill (1,9) circle (3pt);   \node[above]          at (1,9)   {$\bar p_3$};
  \fill (3,7.5) circle (3pt); \node[right]          at (3,7.5) {$\bar p_4$};
  \fill (0.6506,5.8554) circle (3pt); \node[right]  at (1,5.5) {$\bar q$};

  \node[text=orange!80!black] at (-0.7,3.5) {$\bar\triangle_{12}$};
  \node[text=blue!80!black] at (-0,6.3)  {$\bar\triangle_{23}$};
  \node[text=teal!80!black] at (1.6,7.3) {$\bar\triangle_{34}$};
  \node[text=Amber!80!black] at (1.3,4.8)  {$\bar\triangle_{14}$};

\end{scope}
\end{tikzpicture}
\end{center}
\caption{The four subtriangles in the Minkowski plane.}
\label{fig: planar quadrangle four subtriangles}
\end{figure}

\underline{Case 1:} If these subtriangles are adjacent or equal (i.e.\ $j=k$ or $l=i$ etc), they both lie in one of $\bar\triangle_m$ and we already know $\tau(\bar x,\bar y)=\tau(f(\bar x),f(\bar y))$ and $\bar x\leq\bar y\Leftrightarrow f(\bar x)\leq f(\bar y)$ as either $f_{13}$ or $f_{24}$ restrict to an isometry on $\bar\triangle_m$. 

\underline{Case 2:} Now, assume that $\bar x\leq\bar y$ and $[\bar x,\bar y]$ does not contain $\bar q$. Then $[\bar x,\bar y]$ changes subtriangles twice, switching into an adjacent subtriangle of the previous one, say in the distinct points $\bar z \leq \bar w$. By applying case 1 to $\bar x\leq \bar w$ and $\bar z\leq\bar y$, we get that $f$ maps both overlapping segments $[\bar x,\bar w]$ and $[\bar z,\bar y]$ isometrically, in particular, $f$ preserves the length and $f([\bar x,\bar y])$ is a local geodesic. By \cite[Remark 2.6]{erös-gieger2025}, $[f(\bar x),f(\bar y)]$ is in fact a (global) geodesic, therefore $\tau(\bar x,\bar y)=\tau(f(\bar x),f(\bar y))$ and $f(\bar x)\leq f(\bar y)$, as required.

\underline{Case 3:} 
Suppose we are in the case where $\bar x\leq\bar y$ but $[\bar x,\bar y]$ does contain $\bar q$. Then we find $\bar x_n\to\bar x$ such that $\bar x_n\leq \bar y$ and $[\bar x_n,\bar y]$ doesn't contain $\bar q$. Thus, $\tau(\bar x_n,\bar y)=\tau(f(\bar x_n),f(\bar y))$ and $f(\bar x_n)\leq f(\bar y)$ by the previous cases. By continuity of $\tau$ and closedness of $\leq$, we also get $\tau(\bar x,\bar y)=\tau(f(\bar x),f(\bar y))$ and $f(\bar x)\leq f(\bar y)$, as required.

\underline{Case 4:} Finally, we have to exclude the case $f(\bar x)\leq f(\bar y)$ but $\bar x\not\leq\bar y$. So let $\bar\alpha$ be a future-directed timelike curve from $\bar y$ to $\bar p_3$. Note we do know that $\bar x\leq \bar\alpha(1)$, and that $\tau(f(\bar x),f(\bar\alpha(t)))\geq\tau(f(\bar x),f(\bar y))+\tau(f(\bar y),f(\bar\alpha(t)))>0$. Let $t>0$ be such that $\bar x\leq\bar\alpha(t)$ are null related, i.e.\ $\tau(\bar x,\bar\alpha(t))=0$. By the above cases, we know that also $\tau(f(\bar x),f(\bar\alpha(t)))=0$, contradicting our previous conclusion that this is $>0$, so this case cannot happen.
\end{proof}

\section{Parallel lines and rays}
Let us recall the following definition of complete timelike lines (cf.\ \cite[Definition 4.1]{BORS23}).
\begin{dfn}[Complete timelike lines]
Let $X$ be a \LpLSn. A \emph{complete timelike line} in $X$ is a $\tau$-arclength parametrized (thus timelike) curve $\alpha\colon \mathbb{R}\to X$ which is a geodesic between any of its points. A \emph{timelike ray} in $X$ is defined analogously but replacing $\mathbb{R}$ with $[0,\infty)$ or $(-\infty,0]$ for the domain of such curve.
\end{dfn}

We propose the following definition of parallel complete timelike lines, analogous to \cite[Definition 10.5.3]{burago-burago-ivanov2001}. 

\begin{dfn}[Synchronised parallel lines]
\label{dfn:synchronisedparallel}
Let $X$ be a \LpLSn. Two complete timelike lines $\alpha,\beta:\R\to X$ are called \emph{synchronised parallel} if there is a $c\geq0$ such that for all $s<t$
\[\alpha(s)\leq\beta(t) \Leftrightarrow t-s\geq c\]
and, if this is the case, we have
\[
\tau(\alpha(s),\beta(t))=\sqrt{(t-s)^2-c^2} \, ,
\]
and analogously for $\alpha,\beta$ interchanged.
The constant $c$ is called the \emph{spacelike distance} between $\alpha,\beta$, cf. \cite{BORS23}.
\end{dfn}

The following proposition is straightforward and establishes the equivalence between \Cref{dfn:synchronisedparallel} and \cite[Definition~4.16]{BORS23}.
\begin{pop}[Characterising synchronised parallel lines]
\label{pop: syncparallellines}
Let $X$ be a Lorentzian pre-length space with continuous time separation $\tau$ satisfying $\tau(x,x) = 0$ for all $x\in X$. Let $\alpha,\beta\colon \R\to X$ be two complete timelike lines. Then $\alpha$ and $\beta$ are synchronised parallel if and only if there is a $\tau$- and causality-preserving map $f:(\alpha(\R)\cup \beta(\R))\to \R^{1,1}$ of the form $f(\alpha(t))=(t,0)$ and $f(\beta(t))=(t,c)$ for some $c\geq 0$. 
\end{pop}

We also have a notion of parallel timelike lines and rays in the spirit of \cite[Section 9.2.4]{burago-burago-ivanov2001} and \cite[Chapter II.8]{bridson-haefliger1999}

\begin{dfn}[Weakly parallel lines and rays]
\label{def:weakly parallel}
Let $X$ be a \LpLSn. 
Let $\alpha$ and $\beta$ be either two future-directed lines or rays (parametrized on $\R$ or $[0,\infty)$, respectively). 
We say that $\alpha$ is \emph{causally before} $\beta$, and denote it $\alpha\leq\beta$, if there exists an $s\geq 0$  such that for all $t$ in the domain of $\alpha$ and $\beta$ we have $\alpha(t)\leq\beta(t+s)$. 
We than call $\alpha$ and $\beta$ \emph{weakly parallel} if $\alpha\leq\beta\leq\alpha$.
\end{dfn}

\begin{rem}[Requiring timelike instead of causal relations]
In the previous definition we could equivalently ask for the existence of $s\geq 0$ such that $\alpha(t)\ll \beta(t+s)$ and $\beta(t)\ll \alpha(t+s)$ since $\alpha$ and $\beta$ are timelike and Lorentzian pre-length spaces satisfy the push-up property \cite[Lemma 2.10]{kunzinger-saemann2018}. 

Further note that \Cref{def:weakly parallel} can be reformulated similarly to the metric case: it is equivalent to the existence of $s\geq 0$, $\lambda>0$ such that for any $t\in [0,\infty)$, $\tau(\alpha(t),\beta(t+s))>\lambda$. 
\end{rem}

\begin{lem}[Weak parallelity as a relation]
\label{lem:weak parallel transitive}
Being weakly parallel defines an equivalence relation on the set of future-directed, complete timelike lines in $X$.
\end{lem}
\begin{proof}
Symmetry follows by definition. Reflexivity and transitivity follow from reflexivity and transitivity of the causal relation and adding the respective parameters.
\end{proof}

\begin{lem}[Convergence in $d$-arclength implies convergence in $\tau$-arclength]\label{lem:conv_d_vs_tau}
Let $X$ be a \LpLS with continuous $\tau$. Assume that $X$ is $d$-compatible or non-totally imprisoning. Let $\gamma_n:[0,c_n]\to X$ be a sequence of timelike geodesics parametrized by $\tau$-arclength. Suppose the reparametrizations to $d$-arclength converge locally uniformly to a timelike curve $\gamma$ (which is then a geodesic). Then also the original curves converge pointwise to the $\tau$-arclength parametrization of $\gamma$ on the domain of the target.
\end{lem}
\begin{proof}
As $L_{\tau}$ is upper semi-continuous and $\tau$ is lower semi-continuous (even continuous by the curvature bound), we infer that $\gamma$ is a geodesic. 

Now, we denote $\gamma$ as the $\tau$-arclength parametrized version. 
Let $p_n=\gamma_n(0)$, $p=\gamma(0)$. 
Let $f_n$ denote the reparametrizations into $d$-arclength and $f$ the reparametrization to the parametrization given in the limit  -- these are continuous, monotonous and injective. 
We also have that $f_n$ is surjective onto the domain of $\gamma_n$: If $X$ is $d$-compatible, take $t$ in the domain of $\gamma_n$, then we can cover the image of $\gamma_n|_{[0,t]}$ by finitely many $d$-compatible neighbourhoods, showing that the $d$-length of $\gamma_n|_{[0,t]}$ is finite. If $X$ is non-totally imprisoning, $\gamma_n([0,t])$ is compact and thus the $d$-length of $\gamma_n|_{[0,t]}$ is finite.\footnote{For $f$, surjectivity holds by definition.}
Then we have $\gamma_n \circ f_n \to \gamma \circ f$ locally uniformly and we want to show $\gamma_n(s) \to \gamma(s)$ for all $s$ in the interior of the domain of $\gamma$. We have $\tau(p,\gamma(f(t)))=\lim_{n\to\infty}\tau(p_n,\gamma_n(f_n(t)))$, thus $L_\tau(\gamma)\leq\liminf_{n\to\infty}L_\tau(\gamma_n)$, i.e.\ at any parameter in the interior of the domain of $\gamma$ all but finitely many $\gamma_n$ are defined. Let $s$ be arbitrary in the interior of the domain of $\gamma$, then there exist $t$ and $t_n$ such that $s=f(t)=f_n(t_n)$. 
We claim that $t_n \to t$. Indeed, let $\varepsilon > 0$ arbitrary. 
Then by convergence in $d$-arclength, $\gamma_n(f_n(t-\varepsilon))\to\gamma(f(t-\varepsilon))$ and by continuity of $\tau$,
\[f_n(t-\varepsilon)=\tau(p_n,\gamma_n(f_n(t-\varepsilon)))\to\tau(p,\gamma(f(t-\varepsilon)))=f(t-\varepsilon)\]
and so $t-\varepsilon < f_n^{-1}(s)=t_n$ (for large enough $n$). 
Similarly, $f_n^{-1}(s) < t + \varepsilon$ (for large enough $n$), yielding $\lim_n t_n=\lim_nf_n^{-1}(s) =t$. 

Clearly, we have 
\[
d(\gamma_n(s),\gamma(s)) \leq d(\gamma_n(s),\gamma(f(f_n^{-1}(s)))) + d(\gamma(f(f_n^{-1}(s))),\gamma(s)) \, .
\] 
Concerning the first term, we have 
\[d(\gamma_n(s),\gamma(f(f_n^{-1}(s)))) = d(\gamma_n(f_n(t_n)),\gamma(f(t_n))) \to 0\] 
as $\gamma_n \circ f_n \to \gamma \circ f$ locally uniformly and $t_n$ is bounded as $t_n \to t$. 
Concerning the second term, we make use of $f(f_n^{-1}(s)) \to s$, which in turn is implied by continuity of $f$ and $f_n^{-1}(s) \to f^{-1}(s)$, i.e., by $t_n \to t$. 
\end{proof}

\begin{lem}[Timelike asymptotes are rays]\label{lem:tl_asy_are_rays}
Let $X$ be a globally hyperbolic, locally causally closed, \LpLS with timelike curvature globally bounded above by $0$, and proper $d$. 
Let $\gamma$ be a future-directed timelike ray, $x\in I^-(\gamma)$, and $\eta$ the $\tau$-arclength parametrized limit curve of the $d$-arclength parametrized sequence $[x,\gamma(t_n)]$, for some $t_n\to \infty$. Then $L_\tau(\eta)$ is either $0$ (i.e.\ $\eta$ is null) or infinite (i.e.\ $\eta$ is a ray). 
\end{lem}
\begin{proof}
Without loss of generality, assume $x\ll\gamma(0)$. Let $t_n\to \infty$ and $\beta_n = [x,\gamma(t_n)]$. Let us assume for the sake of contradiction that $0<L_\tau(\eta) = t_0 <\infty$. By \Cref{lem:conv_d_vs_tau},  $\beta_n(s)\to \eta(s)$ for any $s\in [0,t_0)$. Moreover, by triangle comparison on $\triangle(x,\gamma(0),\gamma(t_n))$, 
\[
\tau(\beta_n(s), \gamma(s)) \geq \tau_{\mathbb{R}^{1,1}}\left(\overline{\beta_n(s)},\overline{\gamma(s)}\right)
\]
where $\overline{\beta_n(s)}$ and $\overline{\gamma(s)}$ are comparison points for $\beta_n(s), \gamma(s)$ in a comparison triangle for $\triangle(x,\gamma(0),\gamma(t_n))$ in $\mathbb{R}^{1,1}$, for $n$ sufficiently large such that $s<t_n$. However, by the reverse triangle inequality, $\tau(x,\gamma(t_n)) \geq \tau(\gamma(0),\gamma(t_n)) = t_n$, so setting $s':= s t_n/\tau(x,\gamma(t_n))$ we have $s' \leq s$. Therefore,
\[
\tau_{\mathbb{R}^{1,1}}\left(\overline{\beta_n(s)},\overline{\gamma(s)}\right) \geq \tau_{\mathbb{R}^{1,1}}\left(\overline{\beta_n(s)},\overline{\gamma(s')}\right)
\]
and by the intercept theorem,
\[
\tau_{\mathbb{R}^{1,1}}\left(\overline{\beta_n(s)},\overline{\gamma(s')}\right) = \left(1 - \frac{s}{t_n}\right)\tau(x,\gamma(0)) \, .
\]
Thus 
\[
\tau(\beta_n(s), \gamma(s)) \geq  \left(1 - \frac{s}{t_n}\right)\tau(x,\gamma(0)) \, .
\]
By letting $n\to \infty$, we get
\[
\tau(\eta(s), \gamma(s)) \geq \tau(x,\gamma(0)) > 0
\]
implying that $\eta(s) \ll \gamma(s)$. Since $\gamma(s) \ll \gamma(t_0)$ for all $s\in [0,t_0)$, then $\eta(s) \ll \gamma(t_0)$. It follows that $\eta([0,t_0))\subseteq J(x,\gamma(t_0))$ and by global hyperbolicity, $\eta$ must have finite $d$-length. 

Now, if $L_d(\beta_n)\to \infty$ then, by \cite[Theorem~3.14]{kunzinger-saemann2018}, it follows that $\eta$ is inextendible, contradicting that $\eta$ has finite $d$-arclength.

On the other hand, if $L_d(\beta_n) \not\to \infty$ up to passing to a subsequence we can assume that $L_d(\beta_n)$ is bounded. By \cite[Theorem~3.7]{kunzinger-saemann2018} and upper semi-continuity of $L_\tau$ (applied to $1$-Lipschitz reparametrizations of $\beta_n$ and $\eta$ on a common, compact domain), this implies $L_\tau(\eta) \geq \lim_{n\to\infty} L_\tau(\beta_n) = \infty$, which contradicts $\eta\in J(x,\gamma(t_0))$.
\end{proof}

\begin{pop}[Existence and uniqueness of weakly parallel rays]
\label{pop: existence and uniqueness of weakly parallel rays}
Let $X$ be a locally causally closed, regular \LpLS with timelike curvature globally bounded above by $0$. 
Assume that $X$ is either $d$-compatible or non-totally imprisoning. 
Suppose that the metric $d$ is proper. 
Let $\alpha$ be a future-directed ray and $p\in I^-(\alpha)$. 
Then there exists a unique (future-directed) ray $\beta$ weakly parallel to $\alpha$ starting at $p$. 
\end{pop}
\begin{proof}
\begin{figure}
\centering
\begin{tikzpicture}[line cap=round,line join=round,>=triangle 45,x=0.4cm,y=0.4cm]
\clip(-2.,-0.5) rectangle (6.5,16.);
\draw [line width=0.5] (0.,0.)-- (4.,4.);
\draw [line width=0.5] (4.,4.) -- (4.,15.);
\draw [line width=0.5, dotted] (4.,15.) -- (4.,16.);
\draw [line width=0.5] (0.,0.) -- (0.,15.);
\draw [line width=0.5, dotted] (0.,15.) -- (0.,16.);
\draw [line width=0.5] (0.,0.)-- (4.,8.);
\draw [line width=0.5] (0.,0.)-- (4.,14.);
\draw [line width=0.5] (0.,0.)-- (2.5,15.);
\draw [line width=0.5, dotted] (2.5,15.)-- (4.,24.);
\begin{scriptsize}
\draw [fill=black] (0.,0.) circle (1.5pt);
\draw[color=black] (-0.5634124823322493,0) node {$p$};
\draw [fill=black] (4.,4.) circle (1.5pt);
\draw[color=black] (5.2,4) node {$\alpha(t_-)$};
\draw[color=black] (2.5,1.85) node {$\beta_0$};
\draw [fill=black] (4.,14.) circle (1.5pt);
\draw[color=black] (5.2,14) node {$\alpha(t_2)$};
\draw[color=black] (4.5,10.391434358586004) node {$\alpha$};
\draw[color=black] (-0.6020971201847769,8) node {$\beta$};
\draw [fill=black] (4.,8.) circle (1.5pt);
\draw[color=black] (5.2,8) node {$\alpha(t_1)$};
\draw[color=black] (2.5,4.027811431845205) node {$\beta_1$};
\draw[color=black] (2.5,7.199951735752474) node {$\beta_2$};
\draw[color=black] (1.3,10.933019288521391) node {$\beta_3$};
\end{scriptsize}
\end{tikzpicture}
\caption{Construction of the weakly parallel ray $\beta$.}
\label{fig:constructingbeta}
\end{figure}
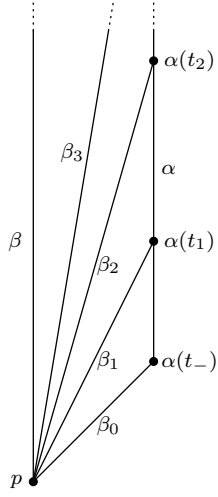
The structure of the proof is as follows. First we prove uniqueness. 
Then we construct a candidate for a timelike ray weakly parallel to $\alpha$ and starting at $p$. This is done by applying the limit curve theorem and \Cref{lem:conv_d_vs_tau}. Crucially, the regularity assumption guarantees the causal character of the resulting ray $\beta$. Then, we prove that $\beta$ is weakly parallel to $\alpha$. This morally follows by observing that comparison triangles of the form $\triangle (\overline{p},\overline{\alpha(0)},\overline{\alpha(t)})$ can be arranged inside a strip in the Minkowski plane bounded by two parallel, future-directed rays and a hyperbola. This eventually leads to the definition of an appropriate constant $d$ such that $\alpha(s)\ll \beta(s+d)$ and $\beta(s)\ll \alpha(s+d)$. 

Concerning uniqueness, suppose there are two rays $\beta_1$ and $\beta_2$ weakly parallel to $\alpha$ through $p$. 
By \Cref{lem:weak parallel transitive}, it follows that also $\beta_1$ and $\beta_2$ are weakly parallel to each other. 
Set $\theta(s,t) = \tma_p(\beta_1(s),\beta_2(t))$ and let $c > 0$ be such that $\beta_1(s) \ll \beta_2(s+c)$ for all $s>0$. 
Note that $\theta$ is non-decreasing due to monotonicity comparison. 
Consider a comparison triangle in Minkowski space for the triangle $\triangle(p, \beta_1(s), \beta_2(s+c))$. 
Then $0 < \tau(\beta_1(s),\beta_2(s+c)) =\tau(\overline{ \beta_1(s)},\overline{ \beta_2(s+c)})$. 
Expressing the timelike relation between two points in the Minkowski plane given in polar coordinates, we obtain $\overline{\beta(s)} \ll \overline{\beta(s+c)} \iff \frac{s+c}{s} > \exp(\theta(s,s+c))$. 
However, $\frac{s+c}{s} \to 1$ as $s \to \infty$, and hence $\theta(s,s+c) \to 0$. 
The non-decreasing property of $\theta$ thus forces $\theta \equiv 0$. In other words, all comparison triangles are degenerate. 
Thus, equality of the rays can be concluded as follows, using strong causality. 
Indeed, given arbitrary $t < s$ and $\varepsilon > 0$, by the degeneracy of the comparison triangle for $\triangle(p, \beta_1(s), \beta_2(s+c))$, we clearly have $\overline{\beta_2(s-\varepsilon)} \ll \overline{\beta_1(s)} \ll \overline{\beta_2(s+\varepsilon)}$. 
By triangle comparison, we infer $\beta_2(s-\varepsilon) \ll \beta_1(s) \ll \beta_2(s+\varepsilon)$. 
It follows that $\{I(\beta_2(t-\varepsilon), \beta_2(t+\varepsilon)) \mid \varepsilon > 0 \}$ is a neighbourhood basis for $\beta_1(t)$, hence $\beta_2(t)=\beta_1(t)$.

Let us now prove existence. 
As $p\in I^-(\alpha)$, there exists $t_-$ such that $p\ll \alpha(t)$ for all $t\geq t_-$. By the global curvature bound, we find geodesics from $p$ to any such $\alpha(t)$. 
Since $X$ is either $d$-compatible or non-totally imprisoning, we can reparametrize geodesics defined on compact intervals with respect to $d$-arclength and apply the limit curve theorem \cite[Theorem 3.14]{kunzinger-saemann2018} to obtain a sequence $t_n \to \infty$ such that the $d$-arclength parametrizations of $\beta_n:=[p,\alpha(t_n)]$ converge locally uniformly to a causal curve $\beta$ starting at $p$, see \Cref{fig:constructingbeta}.   
As $L_{\tau}$ is upper semi-continuous and $\tau$ is lower semi-continuous, we infer that $\beta$ is a geodesic, although it might be null a priori. 

We first show that $\beta$ is timelike. 
Without loss of generality, suppose that $t_-=0$. 
Let $z=\alpha(0)$ and $z_n=\alpha(t_n)$ for $n\geq1$.
Then we have $\widetilde\measuredangle_p(z,z_n) + \widetilde\measuredangle_{z_n}(p,z) = \widetilde\measuredangle_z(p,z_n)$ by \Cref{lem:angleSumMink}. 
As $\widetilde\measuredangle_z(p,z_n)$ is decreasing in $n$ using the monotonicity formulation of timelike curvature bounded above, we infer that $\widetilde\measuredangle_p(z,z_n)$ is bounded. 
Choose a sequence $y_n \in [p,z_n]$ that converges to some point $y \neq p$ on $\beta$ (e.g.\ by choosing $y_n = \beta_n \circ f_n(\varepsilon)$ in the notation of \Cref{lem:conv_d_vs_tau}). 
Again by monotonicity comparison, we get $\widetilde \measuredangle_p(z,z_n) \geq \widetilde \measuredangle_p(z,y_n)$, so this angle is bounded as well. 
Using the Lorentzian law of cosines applied to the comparison triangle of $\triangle(p,y_n,z_n)$, we compute that either $\tau(p,y_n)$ stays bounded away from zero, or that $\tau(y_n,z_n) \to \tau(p,z_n)$. 
In the latter case, we would have $p \leq y \ll z$, $p \neq y$ and $\tau(p,z)=\tau(y,z)$, a contradiction to the regularity of $X$. 
Thus, $\tau(p,y_n)$ is bounded away from zero, which by the continuity of $\tau$ makes $\beta$ initially timelike and by regularity timelike on the entire domain. 
An application of \Cref{lem:tl_asy_are_rays} shows that $\beta$ is a timelike ray.

Now we show that $\beta$  is weakly parallel to $\alpha$. 
Let $a=\tau(p,z), b_n=\tau(z,z_n), c_n=\tau(p,z_n), \theta_n=\tma_{z_n}(p,z)$ and $\omega_n=\tma_{z}(p,z_n)$. 
Then $b_n, c_n \to \infty$ and $\omega_n$ is decreasing.
We can reformulate the Lorentzian law of cosines in $\triangle(\bp_n,\bz,\bz_n)$ to obtain 
\begin{align*}
\frac{a^2 + 2ab_n \cosh(\omega_n)}{c_n+b_n}  & = c_n - b_n \, ,
\end{align*}
which gives that $c_n - b_n$ is bounded. 
Applying the Lorentzian law of cosines in the same triangle but instead using the angle $\theta_n$, we observe that
\begin{align*}
\cosh(\theta_n) & = \frac{b_n^2 + c_n^2 - a^2}{2b_n c_n} \to 1  
\end{align*}
as $c_n - b_n$ is bounded, which shows that $\theta_n \to 0$. 

We can now continue to work explicitly in the model space. 
Denote by $\bp_n$ the comparison points for $p$ in the comparison triangles for $\triangle(p,z,z_n)$ and keep $\bz$ fixed in place in all comparison triangles. 
Moreover, align all comparison triangles in such a way that $[\bz,\bz_n]$ is vertical. 
The idea is to find bounded $d_n$ such that
\[
[\bp_n,\bz_n](s) \ll [\bz,\bz_n](s+d_n) \text{ and } [\bz,\bz_n](s) \ll [\bp,\bz_n](s+d_n)
\]
for all $s$ such that both sides are defined. 
We then show that this is inherited by the limit configuration in $X$. 
The existence of individual values $d_n$ is not an issue, but they may fail to be bounded.
Note that $\bp_n$ is confined to the intersection of a vertical strip with boundary lines through $\bz$ and $\bp_0$ and a hyperbola with centre $\bz$ and radius $\tau(p,z)$. Choose $\bz=(0,0)$, then $[\bz,\bz_n]: s \mapsto (s,0)$ and $[\bp_n,\bz_n]: s \mapsto \bp_n + (\cosh(\theta_n)s,\sinh(\theta_n)s)$ are explicit descriptions of these segments. 
Visually, this means that we have chosen $\bp_n$ to be to the left of the segment $[\bz,\bz_n]$, see \Cref{fig: parallel ray comparison configuration}.
\begin{figure}
\begin{center}
\begin{tikzpicture}
\draw (0,0) -- (0,4);
\draw (0,0) -- (-1.5,-2.5) -- (0,4);
\draw[dotted] (0,0) -- (0,-3);
\draw[dotted] (-1.5,-3) -- (-1.5,4);

\draw[dashed,scale=1,domain=-40:40,smooth,variable=\t] plot ({sqrt(2.5^2-1.5^2)*tan(\t)},{-sqrt(2.5^2-1.5^2)*sec(\t)});

\begin{scriptsize}
\coordinate [circle, fill=black, inner sep=0.5pt, label=0:{$\bz$}] (z) at (0,0);
\coordinate [circle, fill=black, inner sep=0.5pt, label=180:{$\bp_n$}] (pn) at (-1.5,-2.5);
\coordinate [circle, fill=black, inner sep=0.5pt, label=0:{$\bz_n$}] (zn) at (0,4);
\pic [draw, ->, "${\omega_n}$", angle eccentricity=0.5] {angle = zn--z--pn};
\pic [draw, ->, "${\theta_n}$", angle radius=1.2cm, angle eccentricity=1.2] {angle = pn--zn--z};
\end{scriptsize}

\end{tikzpicture}
\end{center}
\caption{The $n$-th step comparison triangle.}
\label{fig: parallel ray comparison configuration}
\end{figure}
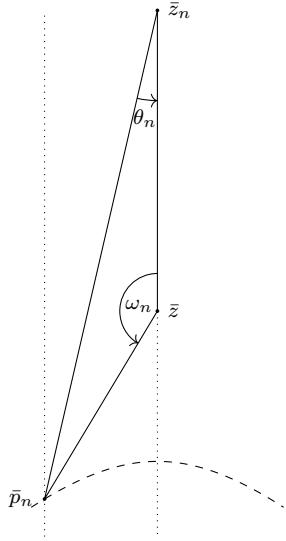

First note that the relation $[\bar p_n,\bar z_n](s)\ll[\bar z,\bar z_n](s+d)$ holds for for any $n\geq 0$ and $d\geq 0$. Indeed the intercept theorem from flat geometry and the relation $\bar p_n\ll\bar z$ imply $[\bar p_n,\bar z_n](\tilde s)\ll[\bar z,\bar z_n](s)$, where $\tilde s$ is chosen such that $\tilde s/L_\tau([\bar p_n,\bar z_n])=s/L_\tau([\bar z,\bar z_n])$. By the reverse triangle inequality the $\tau$-arclength of the curve $[\bar p_n,\bar z_n]$ is longer, hence $s<\tilde s$ and the claim follows. 

Conversely, the desired relation $(s,0)  \ll \bp_n+(\cosh(\theta_n)(s+d_n),\sinh(\theta_n)(s+d_n))$ is true if and only if
\begin{align*}
\cosh(\theta_n)(s+d_n)+t(\bp_n)-s & > -x(\bp_n)-\sinh(\theta_n)(s+d_n) \ [>0] \ & (\iff) \\
\exp(\theta_n)(s+d_n)+t(\bp_n)-s & > -x(\bp_n) \ & (\iff) \\
\exp(\theta_n)d_n & >-x(\bp_n)-t(\bp_n)-s(\exp(\theta_n)-1) \, .
\end{align*}

We notice that $-x(\bp_n)-t(\bp_n)$ is bounded in $n$, and for each $n$ the `worst case' is at $s=0$. This allows us to choose 
\begin{equation*}
d := \sup_n \{ (-x(x_n)-t(x_n))\exp(-\theta_n))\} \, ,
\end{equation*}
which works for all large enough $n$ in the Minkowski plane. 
By the upper curvature bound, this choice of $d$ also works in $X$ for large enough $n$, i.e., $\beta_n(s) \ll \alpha|_{[0,t_n]}(s+d)$ and $\alpha|_{[0,t_n]}(s) \ll \beta_n(s+d)$ for all $s$ where these are defined. 
Letting now $n \to \infty$ gives the claim. 
\end{proof}

In the next section, we will prove that a non-positively curved space which is covered by weakly parallel lines splits as a product. The following lemma and corollary provide an equivalent formulation of 
that assumption, stating that there exists a parallel line through a point if the parallel rays constructed in \Cref{pop: existence and uniqueness of weakly parallel rays} have zero angle. In other words, these parallel rays having positive angle is the only obstruction to the space splitting. 
For example, this is the case in de Sitter space.

\begin{lem}[Concatenation of geodesics with zero angle]\label{lem:concat0angle}
Let $X$ be a \LpLS with timelike curvature globally bounded above by $K$ and let $\alpha:[a,b]\to X,\beta:[b,c]\to X$ be two timelike geodesics with $x=\alpha(b)=\beta(b)$, $\ma_x(\alpha,\beta)=0$ and $\tau(\alpha(a),\alpha(b))+\tau(\beta(b),\beta(c))\leq D_K$. Then the concatenation of $\alpha,\beta$ is a geodesic.
\end{lem}
\begin{proof}
Let $p_-=\alpha(a)$ and $p_+=\beta(c)$. Construct a comparison hinge $(\tilde\alpha,\tilde\beta)$ for $(\alpha,\beta)$: As $\ma_x(\alpha,\beta)=0$, $\tilde\alpha,\tilde\beta$ join to a single geodesic.
By hinge comparison for the hinge $(\alpha,\beta)$, we get that $\tau(p_-,p_+)\leq\tau(\tilde p_-,\tilde p_+)=\tau(\alpha(a),\alpha(b))+\tau(\beta(b),\beta(c))$, the converse is just the reverse triangle inequality, making the concatenation a geodesic.
\end{proof}
\begin{cor}[Parallel rays with zero angle fit together to a line]\label{cor:concat0angle}
Let $X$ be a globally hyperbolic regular \LpLS with timelike curvature globally bounded above by $0$. Let $\alpha$ be a complete timelike line, $p \in I^-(\alpha) \cap I^+(\alpha)$ and $\beta^+$ and $\beta^-$ the unique future and past rays emanating from $p$ and weakly parallel to the half lines $\alpha^+=\alpha|_{[0,\infty)}$ and $\alpha^-=\alpha|_{(-\infty,0]}$. Then the concatenation of $\beta^+$ and $\beta^-$ is a line if and only if $\ma_p(\beta^-,\beta^+)=0$ and in this case it is the unique weakly parallel line to $\alpha$ through $p$.
\end{cor}

\begin{proof}
Note that by \Cref{pop: existence and uniqueness of weakly parallel rays}, we find $\beta^\pm$. If they fit together to a line, it is clear that $\ma_p(\beta^-,\beta^+)=0$. Conversely, if $\ma_p(\beta^-,\beta^+)=0$,  \Cref{lem:concat0angle} implies that they fit together to a timelike line.

Concerning the weak parallelity, we already know that $\alpha^+(t) \leq \beta^+(t+s) \leq \alpha^+(t+2s)$ and $\alpha^-(t) \geq \beta^-(t+s) \geq \alpha^-(t+2s)$ for a suitable choice of $s>0$. Then choosing $2s$ in the definition of weakly parallel lines avoids having to deal with relations between $\alpha^\mp$ and $\beta^\pm$: e.g.\ $\alpha^-(t)\leq \alpha^+(0)\leq\beta^+(t+2s)$ if $-s<t<0$. 

Finally, uniqueness also follows from \Cref{pop: existence and uniqueness of weakly parallel rays}.
\end{proof}

\section{Splitting theorem}
Finally, we will prove our main theorem, i.e. a splitting theorem for Lorentzian pre-length spaces with timelike curvature bounded above by $0$. This result can be seen as a Lorentzian version of a classical splitting theorem for certain metric spaces of non-positive curvature (see e.g.\ \cite[Theorem 9.2.31.]{burago-burago-ivanov2001}). Note that the metric theorem is stated as a result for Hadamard spaces i.e. complete, simply connected spaces with a (local) curvature bound above by $0$. It can be shown (see \cite[Theorem 9.2.9]{burago-burago-ivanov2001}) that in such spaces the curvature bound actually holds globally, i.e. the entire space $X$ is a curvature comparison neighbourhood, a fact the proof of the splitting theorem makes heavy use of. It can also be shown that in the class of complete metric spaces of non-positive curvature the curvature bound holding globally is equivalent to the space being simply connected (see \cite[Remark 9.2.11.]{burago-burago-ivanov2001}). Therefore assuming the space to be a Hadamard space and assuming a global curvature bound are essentially the same. \\

The situation in the Lorentzian setting is very similar. There does not exist a well established notion of Lorentzian Hadamard spaces, but such a definition would probably at least assume the space to have timelike curvature bounded from above by $0$ and be future one-connected. The latter means that between any two future-directed timelike curves $\gamma_0,\gamma_1$ with the same endpoints there exists an endpoint preserving homotopy $\gamma_t$, such that $\gamma_t$ is still future-directed timelike. Indeed it can be shown that if one assumes these two properties as well as additional regularity for the Lorentzian pre-length space (for example a globally hyperbolic and regular Lorentzian length space would be enough) this once again implies a global curvature bound (see \cite[Theorem 4.8]{erös-gieger2025}). In fact, analogously to the metric case it can be shown that in the class of globally hyperbolic and regular Lorentzian length spaces with timelike curvature bounded above by $0$ the curvature bound holding globally is equivalent to the space being future one-connected.

\begin{pop}[Non-positive curvature and future one-connectedness]
    Let $X$ be a globally hyperbolic and regular Lorentzian length space with timelike curvature globally bounded above by $0$. Then $X$ is future one-connected.
\end{pop}
\begin{proof}
    Let $\gamma_0,\gamma_1:[0,1]\to X$ be two future-directed timelike curves, such that $p:=\gamma_0(0)=\gamma_1(0)$ and $q:=\gamma_0(1)=\gamma_1(1)$. If we can find an endpoint preserving timelike homotopy between $\gamma_0$ and $\gamma_1$ we have proven that $X$ is future one-connected. By \cite[Proposition 4.6]{erös-gieger2025} there exist future-directed timelike geodesics $\alpha_1,\alpha_2$ from $p$ to $q$ such that $\gamma_i$ is timelike homotopic to $\alpha_i$ via an endpoint preserving homotopy. Moreover, by \cite[Proposition 2.3]{erös-gieger2025}, between any two timelike related points $p,q$ in a ($\leq 0$)-comparison neighbourhood (such as $X$) there exists, up to reparametrization, a unique timelike geodesic. This implies that $\alpha_0=\alpha_1$ and in particular there exists an endpoint preserving timelike homotopy between $\gamma_0$ and $\gamma_1$.
\end{proof}

While the above proposition proves that for a large class of spaces it does not matter whether one assumes a local curvature bound plus future one-connectedness or a global curvature bound, this equivalence only holds under additional assumptions on the space. Therefore, to weaken our assumptions, we will henceforth just assume $X$ to be a Lorentzian pre-length space with timelike curvature globally bounded above by $0$. 

\begin{lem}[Weakly parallel lines have constant $\tau$-distance]
\label{lem:constancy of time separation}
Let $\alpha_1, \alpha_2$ be two future-directed, weakly parallel complete timelike lines in $X$ and let $c\in \mathbb{R}$. Then $t\mapsto\tau(\alpha_1(t),\alpha_2(t+c))$ is constant.
\end{lem}
\begin{proof}
Since $X$ has timelike curvature globally bounded above by $0$, the proof of \cite[Proposition 6.1]{beran-kunzinger-rott2024} implies that the map $t\mapsto \tau(\alpha_1(t),\alpha_2(t+c))$ is concave where positive. Let $c>0$ be chosen such that $\alpha_1(t)\ll\alpha_2(t+c)$ for all $t\in\mathbb R$, then $t\mapsto \tau(\alpha_1(t),\alpha_2(t+c))$ is a concave and positive function on $\R$ and hence constant. If $c$ is chosen such that $\alpha_1(t)\not\ll \alpha_2(t+c)$ for all $t\in\R$, then the map is constant $0$. It remains to show that all $c\in\R$ fall into one of these two scenarios. 

Suppose to the contrary, that there exist parameters $t_0,c_0\in\mathbb R$ such that 
\[\tau(\alpha_1(t_0),\alpha_2(t_0+c_0))=\varepsilon>0\]
but $\alpha_1(t)\ll \alpha_2(t+c_0)$ does not hold for all $t\in\mathbb R$. We denote the set of all $c$ such that $\tau(\alpha_1(t),\alpha_2(t+c))=:F(c)$ is constant and positive as $\mathcal C$. Note that this means $c_0<c$ for all $c\in\mathcal C$. Moreover, the map $\mathcal{C}\ni c\mapsto F(c)$ is continuous and becomes arbitrarily small by continuity of the time separation function. In particular, there exists a $c\in\mathcal C$ such that $F(c)<\varepsilon$. But then, 
\begin{align*}
    \tau(\alpha_1(t_0),\alpha_2(t_0+c_0))&>\tau(\alpha_1(t_0),\alpha_2(t_0+c))\\
    &\geq\tau(\alpha_1(t_0),\alpha_2(t_0+c_0))+\tau(\alpha_2(t_0+c_0),\alpha_2(t_0+c)) \, ,
\end{align*}
which is a contradiction.
\end{proof}

\begin{rem}[On concavity]
Observe the proof of \cite[Proposition 6.1]{beran-kunzinger-rott2024} works in the setting of the previous lemma, without the assumption of the strict causal triangle comparison condition. This is because, adapting the proof to our setting, the only part where the strict causal triangle comparison condition is used is to verify that, given a timelike geodesic $\gamma\colon [0,1]\to X$ between $\alpha_1(t_1)$ and $\alpha_2(t_2+c)$ with $t_1<t_2$, the causal relation 
\[
\alpha_1(st_2+(1-s)t_1)\le \gamma(s)\le \alpha_2(st_2+(1-s)t_1+c)
\]
holds.
However, in a previous step of the same proof, it is shown that 
\[
\tau(\alpha_1(st_2+(1-s)t_1),\gamma(s))\geq s\tau(\alpha_1(t_2),\alpha_2(t_2+c))
\]
and 
\[
\tau(\gamma(s),\alpha_2(st_2+(1-s)t_1+c))\geq (1-s)\tau(\alpha_1(t_1),\alpha_2(t_1+c)) \, ,
\]
and since $\alpha_1(t)\ll \alpha_2(t+c)$ for any $t\in \mathbb{R}$, the required causal (actually chronological) relations hold.
\end{rem}

As in the proof of \Cref{lem:constancy of time separation}, for any $c\in\mathbb{R}$ such that $\alpha_1(t)\ll \alpha_2(t+c)$ for all $t\in\mathbb{R}$, we will henceforth denote by $F(c)$ the constant value of the map $t\mapsto\tau(\alpha_1(t),\alpha_2(t+c))$.

\begin{lem}[Weakly parallel lines have constant angles]
\label{lem:constancy of angle}
Under the same hypothesis of \Cref{lem:constancy of time separation}, let $d\neq 0$ and $c\in \mathbb{R}$ such that $\alpha_1(t-c)$ and $\alpha_2(t)$ are timelike related for all $t\in\mathbb{R}$. Then the map
\[
t\mapsto \ma_{\alpha_{2}(t)} (\alpha_1(t-c),\alpha_2(t+d))
\]
is constant.
\end{lem}

\begin{proof}
First, observe that the angle $\ma_{\alpha_{2}(t)} (\alpha_1(t-c),\alpha_2(t+d))$ is finite for any $t\in \mathbb{R}$. Indeed, if the sign of the angle is $\sigma=-1$, this is a consequence of angle comparison. In the case of $\sigma=1$, this follows from the case $\sigma=-1$ combined with the inequality
\[
\ma_{\alpha_2(t)}(\alpha_1(t-c),\alpha_2(t+d)) \leq \ma_{\alpha_2(t)}(\alpha_1(t-c),\alpha_2(t-d))
\]
obtained by \Cref{pop:triangle inequality angles} \ref{triIneqAng:AlongGeo}. For the rest of the proof we will only treat the case $d>0$ and $\alpha_1(t-c)\ll\alpha_2(t)$, with the remaining cases working analogously.

Our curvature assumption allows us to apply \Cref{fvf} without extra assumptions (see \Cref{rem:assumptions on fvf}).  Thus,
\begin{equation}\label{eq:fvf constant angle}
\cosh(\ma_{\alpha_{2}(t)} (\alpha_1(t-c),\alpha_2(t+d))) = \left.\frac{d}{ds}\right|_{s=0^+} \tau(\alpha_1(t-c),\alpha_2(t+s)) \, .
\end{equation}
Since $\alpha_1(t-c) \ll \alpha_2(t)\ll \alpha_2(t+s)$ for any $s>0$, we have
\[
\tau(\alpha_1(t-c),\alpha_2(t+s)) = F(c+s) \, .
\]
In particular, the right hand side of equation~\eqref{eq:fvf constant angle} equals 
\[
\left.\frac{d}{ds}\right|_{s=0^+} F(c+s) =: F'(c^+) \, ,
\]
which does not depend on $t$, and the claim follows. 
\end{proof}

\begin{lem}[Weakly parallel lines bound flat strip]
\label{lem:flat strip}
Assume the same hypothesis as in \Cref{lem:constancy of time separation} and, additionally, that $X$ is globally causally closed. 
Then either $\alpha_1(\mathbb{R})=\alpha_2(\mathbb{R})$ or they span a timelike convex flat strip, i.e., there are parallel lines $\widetilde{\alpha}_1,\ \widetilde{\alpha}_2$ in $\mathbb{R}^{1,1}$ and an isometry $f\colon\conv(\widetilde{\alpha}_1(\mathbb{R})\cup\widetilde{\alpha}_2(\mathbb{R}))\to X$ such that $f\circ \widetilde{\alpha}_i=\alpha_i$, $i=1,2$. In particular, up to a shift reparametrization, $\alpha_1,\alpha_2$ are synchronised parallel.
\end{lem}

\begin{proof}
Fix $c\in\mathbb{R}$ such that $\alpha_1(t)\ll\alpha_2(t+c)$ and $\alpha_2(t) \ll \alpha_1(t+c)$ for any $t\in \mathbb{R}$. Moreover, since $\alpha_1$ and $\alpha_2$ are both future-directed, we can assume that $c\geq 0$. Let $d\geq 0$ and, for any $t\in \mathbb{R}$, let
\begin{align*}
p_1(t) &= \alpha_1(t) \, , \\
p_2(t) &= \alpha_2(t+c) \, , \\
p_3(t) &= \alpha_2(t+3c+d) \, , \\
p_4(t) &= \alpha_1(t+2c+d) \, .
\end{align*}
Then, by construction, $p_1(t)\ll p_2(t)\ll p_4(t)\ll p_3(t)$, and by \Cref{lem:constancy of angle} and \Cref{pop:triangle inequality angles} \ref{triIneqAng:AlongGeo}, we have
\begin{equation}\label{eq:corresponding angles1}
\ma_{p_1(t)}(p_2(t),p_4(t)) \geq \ma_{p_4(t)}(p_1(t),p_3(t))
\end{equation}
and
\begin{equation}\label{eq:corresponding angles2}
\ma_{p_3(t)}(p_4(t),p_2(t)) \geq \ma_{p_2(t)}(p_1(t),p_3(t)) \, .
\end{equation}
This implies
\begin{equation}\label{eq:quadrangle-splitting theorem}
\begin{split}
\ma_{p_1(t)}(p_2(t),p_4(t))\ + &\ \ma_{p_3(t)}(p_4(t),p_2(t)) \geq\\  &\ma_{p_2(t)}(p_1(t),p_3(t))   +  \ma_{p_4(t)}(p_1(t),p_3(t)) \, .
\end{split}
\end{equation}

\begin{figure}
\begin{tabular*}{\textwidth}{@{}cc@{}}
\begin{minipage}{0.5\textwidth}
\centering
\def \globalscale {1}
\begin{tikzpicture}[y=1cm, x=1cm, yscale=0.5*\globalscale,xscale=0.5*\globalscale, every node/.append style={scale=\globalscale}, inner sep=0pt, outer sep=0pt]
  \begin{scriptsize}
  \fill 
  (5.8,14)   circle (3pt)   node[left,inner sep=4pt]{$p_1(t)$}
  (5.8,23.5) circle (3pt)   node[left,inner sep=4pt]{$p_4(t)$}
  (10,18.5)  circle (3pt)   node[right,inner sep=4pt]{$p_2(t)$}
  (10,28)    circle (3pt)   node[right,inner sep=4pt]{$p_3(t)$};
  
  \fill[black!10] (5.8,14) -- (10,18.5) -- (10,28) -- (5.8,23.5) -- cycle; 
  
  \node[anchor=south west] at (7.0, 20.8){$Q_d(t)$};
  \node[above,inner sep=4pt] at (5.8, 29){$\alpha_1$};
  \node[above,inner sep=4pt] at (10, 29){$\alpha_2$};
  \end{scriptsize}
  \path[draw=black,line width=1.pt,miter limit=4.0] 
  (5.8,14.7)arc(90:40:0.6) 
  (5.8, 24.2)arc(90:40:0.6)
  (10,19.0)arc(90:220:0.6)
  (10,28.5)arc(90:220:0.6);
  
  \path[draw=black,line cap=butt,line join=miter,line width=1.pt] 
  (5.8, 14.0)   -- (10.0, 18.5) 
  (5.8, 23.5)   -- (10.0, 28.0) 
  (5.8, 29.1)   -- (5.8, 13.0)
  (10.0, 29.1)  -- (10.0, 13.0);
  
  \path[draw=black,line cap=butt,line join=miter,line width=1.pt] 
  (6.1, 14.8)   -- (5.9, 14.4)
  (6.1, 24.3)   -- (5.9, 23.9) 
  (9.4, 18.9)   -- (9.8, 18.65)
  (9.3, 18.8)   -- (9.7, 18.55)
  (9.4, 28.4)   -- (9.8, 28.15) 
  (9.3, 28.3)   -- (9.7, 28.05)
  (9.8, 22.9)   -- (10.2, 22.9)
  (5.6, 19.4)   -- (6, 19.4)
  (7.5, 25.6)   -- (7.8, 25.35)
  (7.6, 25.7)   -- (7.9, 25.45)
  (7.5, 16.1)   -- (7.8, 15.85)
  (7.6, 16.2)   -- (7.9, 15.95);
\end{tikzpicture}
\end{minipage}
&
\begin{minipage}{0.5\textwidth}
\centering
\def \globalscale {1}
\begin{tikzpicture}[y=1cm, x=1.3cm, yscale=0.53*\globalscale,xscale=0.53*\globalscale, every node/.append style={scale=\globalscale}, inner sep=0pt, outer sep=0pt]
  \fill[black!10] (2.4,12) -- (4.7,15.6) -- (4.7,27.5) -- (2.4,23.9) -- cycle;
  \fill[black!20] (2.4,14.7) -- (4.7,18.3) -- (4.7,24.7) -- (2.4,21.1) -- cycle;
  \path[draw=black,line width=1.pt] 
  (2.4, 28) -- (2.4, 11.5)
  (4.7, 28) -- (4.7, 11.5)
  (2.4, 23.9) -- (4.7, 27.5)
  (2.4, 21.1) -- (4.7, 24.7)
  (2.4, 14.7) -- (4.7, 18.3)
  (2.4, 12) -- (4.7, 15.6);
\begin{scriptsize}
  \node[text=black,anchor=south west,line width=1.pt]  at (3.03, 24.2){$\overline{Q}_{d'}(t')$};  
  \node[text=black,anchor=south west,line width=1.pt]  at (3.1, 19.6){$\overline{Q}_{d}(t)$};
  \node[text=black,anchor=south west,line width=1.pt]  at (3.55, 26.7){$\vdots$};
  \node[text=black,anchor=south west,line width=1.pt]  at (3.55, 12.4){$\vdots$};    
\end{scriptsize}
\end{tikzpicture}
\end{minipage}
\\
\begin{minipage}[t]{0.5\textwidth}
\captionsetup{width=0.8\linewidth, font=normalsize}
\captionof{figure}{The quadrilateral $Q_{d}(t)$ bounds a convex region.}
\label{fig:quadrilateral1}
\end{minipage}
&
\begin{minipage}[t]{0.5\textwidth}
\captionsetup{width=0.8\linewidth, font=normalsize}
\captionof{figure}{$\overline{Q}_{d}(t)\subseteq \overline{Q}_{d'}(t')$ for $0<t-t'<d'-d$.}
\label{fig:quadrilateral2}
\end{minipage}
\end{tabular*}
\end{figure}

By \Cref{pop:quadrangle rigidity}, the quadrangle $Q_{d}(t)=(p_1(t),p_2(t),p_3(t),p_4(t))$ bounds a convex flat region, see \Cref{fig:quadrilateral1}, i.e.\ there exists a quadrangle 
\[
\overline{Q}_{d}(t)=\conv(\overline{p_1(t)},\overline{p_2(t)},\overline{p_3(t)},\overline{p_4(t)})\subseteq \mathbb{R}^{1,1}
\]
and a $\tau$- and causality-preserving map $f_{d}^t\colon \overline{Q}_{d}(t)\to X$ such that 
\[
f_{d}^t(\overline{p_i(t)})=p_i(t) \, , \quad i=1,2,3,4.
\]
Moreover, the inequality \eqref{eq:quadrangle-splitting theorem} is actually an equality, which in turn implies that inequalities \eqref{eq:corresponding angles1} and \eqref{eq:corresponding angles2} are equalities as well. Since 
\[
\widetilde{\ma}_{p_1(t)}(p_2(t),p_4(t)) = \ma_{p_1(t)}(p_2(t),p_4(t)) = \cosh^{-1}(F'(c^+))
\]
and
\[
\tau_{\mathbb{R}^{1,1}}(\overline{p_1(t)},\overline{p_2(t)}) = \tau(p_1(t),p_2(t)) = F(c)
\]
are independent of $t$ and $d$, then, up to an isometry of the Minkowski plane, we can assume that 
$\overline{Q}_{d}(t)\subseteq \mathbb{R}^{1,1}$ is a parallelogram with vertices of the form 
\begin{align*}
&\overline{p_1(t)}=(t,0) \, ,\\
&\overline{p_2(t)}=(t+F(c)F'(c^+), F(c)\sqrt{F'(c^+)^2-1}) \, ,\\
&\overline{p_3(t)}=(t+2c+d+F(c)F'(c^+),F(c)\sqrt{F'(c^+)^2-1}) \, ,\\
&\overline{p_4(t)}=(t+2c+d,0) \, .
\end{align*}
Therefore, whenever $d'-d>t-t'>0$ we have 
\[
\overline{Q}_{d}(t)\subseteq \overline{Q}_{d'}(t') \, ,
\]
see \Cref{fig:quadrilateral2}. 
Furthermore, by the proof of \Cref{pop:quadrangle rigidity}, if $(x,y)\in \overline{Q}_{d}(t)$ then $f_d^t(x,y) = [\alpha_1(s),\alpha_2(s+c)](\lambda)$ for unique $s \in [t,t+2c+d]$ and $\lambda \in [0,F(c)]$, where $[\alpha_1(s),\alpha_2(s+c)]$ is the unique geodesic joining $\alpha_1(s)$ and $\alpha_2(s+c)$, parametrized by $\tau$-arclength. Therefore, $f^{t'}_{d'}|_{\overline{Q}_{d}(t)} = f^t_{d}$ whenever $0<t-t'<d'-d$. 

We can thus define the $\tau$- and causality-preserving map 
\[
f\colon \mathbb{R}\times [0,F(c)\sqrt{F'(c^+)^2-1}]\to X
\]
by setting
\[
f(x,y) = f_d^t(x,y) 
\]
for any $d>0$ and $t\in \mathbb{R}$ such that $(x,y)\in \overline{Q}_d(t)$. The last assertion of the theorem is an immediate consequence of \Cref{pop: syncparallellines}.
\end{proof}

Finally, we are in a position to prove the main theorem of this work. 

\begin{thm}[Lorentzian splitting theorem for non-positive curvature]\label{thm:lorentzianHadamard}
Let $X$ be a \LpLS with timelike curvature globally bounded above by $0$. 
Let $\gamma$ be a complete timelike line and define
\[
S=\{\alpha:\alpha\text{ is a complete timelike line weakly parallel to }\gamma\}/\sim
\]
where $\alpha\sim\beta$ if there exists $c\in\mathbb{R}$ such that $\alpha(t) = \beta(t+c)$ for all $t\in \mathbb{R}$. Then each element of $S$ has a representative which is synchronised parallel to $\gamma$, and these representatives are synchronised parallel to each other. Moreover, the quantity
\[
d_S([\alpha],[\beta]) := \inf\{(t-s)/2:\beta(s)\leq\alpha(0)\leq\beta(t)\}
\]
does not depend on the chosen representatives and it is the spacelike distance between the synchronised parallel representatives. 

This defines a metric on $S$ such that $(S,d_S)$ is a $\mathrm{CAT}(0)$ space, and the map $f\colon \prescript{-}{}{\mathbb{R}}\times S\to X$ given by $(t,[\alpha]) \mapsto {\alpha}(t)$ is an isometric embedding, where ${\alpha}$ is the unique representative of $[\alpha]$ that is synchronised parallel to $\gamma$.

In particular, if $X=\bigcup_{[\alpha]\in S}\alpha(\R)$ then $X$ is isometric to the product $\prescript{-}{}{\R}\times S$, where $S$ is a $\mathrm{CAT}(0)$ space.
\end{thm}

\begin{proof}
Throughout this proof, we use the notation from \Cref{def: lor product} for Lorentzian products, and we denote by $\leq_{{\mathbb{R}}\times S}$ and $\tau_{{\mathbb{R}}\times S}$ the causal relation and time separation of the Lorentzian product $\prescript{-}{}{\mathbb{R}}\times S$. Analogously, we denote by $\leq_X$ and $\tau_X$ the causal relation and time separation corresponding to $X$. 

Let us assume first that we have proven already that every element of $S$ has a representative that is synchronised parallel to $\gamma$ and these representatives are synchronised parallel to each other. Then \Cref{pop: syncparallellines} and explicit calculations in $\mathbb{R}^{1,1}$ show that the formula of $d_S([\alpha],[\beta])$ in the statement is independent of representatives, equal to the spacelike distance and can be alternatively expressed as 
\begin{equation}\label{eq:alternative defintion d_S}
d_S([\alpha],[\beta]) = \inf\{t-s:\alpha(s)\leq_X\beta(t)\}
\end{equation}
for any $s\in \R$ where $\alpha$ and $\beta$ are the synchronised parallel representatives. 
Then, just as in the proof of \cite[Lemma 5.2]{BORS23}, we obtain that $d_S$ is non-negative, symmetric and positive definite. 

To prove that $d_S$ satisfies the triangle inequality, consider $[\alpha],[\beta],[\eta]\in S$, assume that $\alpha$, $\beta$, $\eta$ are pairwise synchronised parallel, and let $\varepsilon>0$. By 
equation~\eqref{eq:alternative defintion d_S},
\[
\alpha(t)\leq_X \eta(t+d_S([\alpha],[\eta])+\varepsilon)\leq_X \beta(t+d_S([\alpha],[\eta])+d_S([\eta],[\beta])+2\varepsilon) \, ,
\]
for all $t\in\mathbb{R}$. 
Therefore 
\[
\alpha(t)\leq_X \beta(t+d_S([\alpha],[\eta])+d_S([\eta],[\beta])+2\varepsilon)
\]
and analogously
\[
\beta(t)\leq_X \alpha(t+d_S([\alpha],[\eta])+d_S([\eta],[\beta])+2\varepsilon)
\]
for all $t\in \mathbb{R}$. 
Again by definition of $d_S$, this implies that 
\[
d_S([\alpha],[\beta])\le d_S([\alpha],[\eta])+d_S([\eta],[\beta])+2\varepsilon \, ,
\]
and by letting $\varepsilon\to 0$, the triangle inequality follows. 

We also get that $(S,d_S)$ is a geodesic space. Indeed, given $[\alpha_0],[\alpha_1] \in S$, by \Cref{lem:flat strip} there are parallel, future-directed, complete timelike lines $\overline{\alpha}_0$ and $\overline{\alpha}_1$ in $\mathbb{R}^{1,1}$ and an isometry $f\colon \conv(\overline{\alpha}_0(\mathbb{R})\cup\overline{\alpha}_1(\mathbb{R}))\to X$ such that $f\circ\overline{\alpha}_i=\alpha_i$, $i=0,1$. In particular, up to composing with an isometry of $\mathbb{R}^{1,1}$ we can assume that $\overline{\alpha}_0(t) = (t,0)$ and $\overline{\alpha}_1(t) = (t,d_S([\alpha_0],[\alpha_1]))$. Thus we can define $\sigma\colon[0,d_S([\alpha_0],[\alpha_1])]\to S$ by $\sigma(s) = [f\circ \overline{\alpha}_s]$, where $\overline{\alpha}_s\colon \mathbb{R}\to\mathbb{R}^{1,1}$ is given by $\overline{\alpha}_s(t) = (t,s)$, and the fact that $f$ is an isometry implies that $\sigma$ is a geodesic in $S$ between $[\alpha_0]$ and $[\alpha_1]$. 

In order to obtain a $\tau$- and causality-preserving map from $\prescript{-}{}{\mathbb{R}}\times S$ into $X$, observe that for any $(s,[\alpha]),(t,[\beta])\in \mathbb{R}\times S$, we have
\[
(s,[\alpha])\leq_{\mathbb{R}\times S} (t,[\beta])
\]
if and only if $t-s\geq d_S([\alpha],[\beta])$, which in turn is equivalent to 
\[
\alpha(s)\leq_X \beta(t) \, ,
\]
assuming that $\alpha$ and $\beta$ are synchronised parallel representatives. In this case, 
\[
\tau_{\mathbb{R}\times S}((s,[\alpha]),(t,[\beta])) = \sqrt{(t-s)^2-d_S([\alpha],[\beta])^2} = \tau_X(\alpha(s),\beta(t)) \, .
\]
Therefore, the map $(s,[\alpha])\mapsto \alpha(s)$, where $\alpha$ is synchronised parallel to $\gamma$, is a $\tau$- and causality-preserving map from $\prescript{-}{}{\mathbb{R}}\times S$ into $X$. Furthermore, since $X$ has timelike curvature globally bounded above by $0$, then so does $\prescript{-}{}{\mathbb{R}}\times S$. In particular, $S$ is a $\mathrm{CAT}(0)$ space by \Cref{pop:AGKS-products}.

Let us now prove that each element of $S$ has a representative which is synchronised parallel to $\gamma$ and that these representatives are synchronised parallel to each other. The first assertion is an immediate consequence of \Cref{lem:flat strip}. For the second assertion, we proceed along the same lines as the proof of \cite[Theorem 9.2.31]{burago-burago-ivanov2001}.

Let $[\alpha],[\beta]\in S$. Since both $\alpha$ and $\beta$ are weakly parallel to $\gamma$, by  \Cref{lem:flat strip} we can assume they are synchronised parallel to $\gamma$. Therefore, there exist $\tau$- and causality-preserving maps $f_{\alpha,\gamma}\colon (\alpha(\mathbb{R})\cup\gamma(\mathbb{R}))\to \mathbb{R}^{1,1}$ and $f_{\gamma,\beta}\colon (\gamma(\mathbb{R})\cup\beta(\mathbb{R}))\to \mathbb{R}^{1,1}$ such that 
\begin{align*}
&f_{\alpha,\gamma}(\alpha(t))=(t,0) \, ,\\ 
&f_{\gamma,\beta}(\beta(t)) = (t,d_S([\alpha],[\gamma])+d_S([\gamma],[\beta])) \, ,\\
&f_{\alpha,\gamma}(\gamma(t))=f_{\gamma,\beta}(\gamma(t))=(t,d_S([\alpha],[\gamma])) \, .
\end{align*}
It is clear that 
\[
t - \sqrt{t^2-(d_S([\alpha],[\gamma])+d_S([\gamma],[\beta]))^2} \xrightarrow[]{t\to \infty} 0 \, .
\]
Moreover, if $t > d_S([\alpha],[\gamma])+d_S([\gamma],[\beta])$ then, by definition of $\tau_{\mathbb{R}^{1,1}}$, the fact that $f_{\alpha,\gamma}$ and $f_{\gamma,\beta}$ are $\tau$- and causality-perserving, the intercept theorem, and the reverse triangle inequality,
\begin{align*}
\sqrt{t^2-(d_S([\alpha],[\gamma])+d_S([\gamma],[\beta]))^2} &= \tau_{\mathbb{R}^{1,1}}(f_{\alpha,\gamma}(\alpha(0)),f_{\gamma,\beta}(\beta(t)))\\
&=\tau_X(\alpha(0),\gamma(s))+\tau_X(\gamma(s),\beta(t)) \\
&\leq \tau_X(\alpha(0),\gamma(t)) \, ,
\end{align*}
where $s$ is such that $f_{\alpha,\gamma}(\gamma(s))$ is on $[f_{\alpha,\gamma}(\alpha(0)),f_{\gamma,\beta}(\beta(t))]$ (see \Cref{fig:three parallel lines}). Thus, for any $\varepsilon>0$, 
\begin{equation}\label{eq:asymptotics1}
t- \tau_X(\alpha(0),\beta(t)) <\varepsilon
\end{equation}
for sufficiently large $t$.

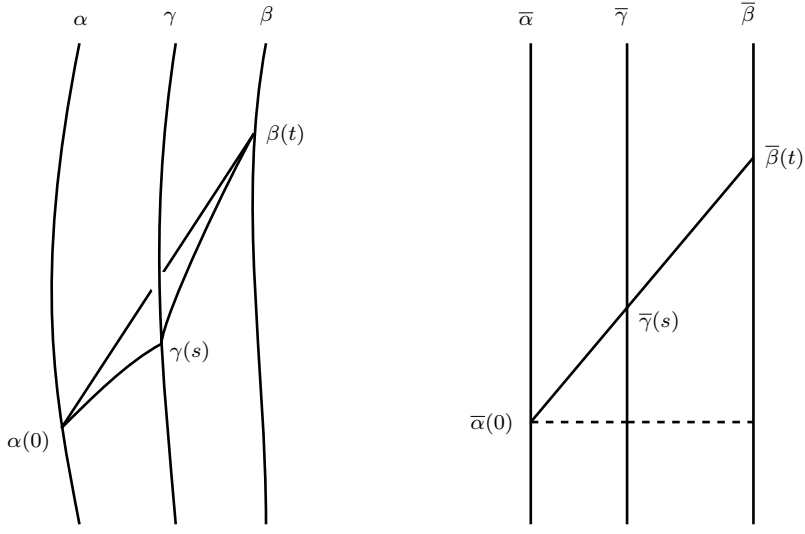
\begin{figure}
\def \globalscale {1}
\begin{tikzpicture}[y=1cm, x=1cm, yscale=0.8*\globalscale,xscale=0.8*\globalscale, every node/.append style={scale=\globalscale}, inner sep=0pt, outer sep=0pt]
  \path[draw=black,line cap=butt,line join=miter,line width=1.pt] 
  (2.7, 21.3) -- (5.9, 26.2)
  (10.5, 27.7) -- (10.5, 19.7)
  (12.1, 27.7) -- (12.1, 19.7)
  (14.2, 27.7) -- (14.2, 19.7)
  (10.5, 21.4) -- (14.2, 25.8)
  (3.0, 27.7).. controls (2.2, 23.7) and (2.6, 21.8) .. (3.0, 19.7)
  (6.1, 27.7).. controls (5.6, 25.0) and (6.1, 22.4) .. (6.1, 19.7);
  
  \path[fill=white,line width=0.1cm,miter limit=4.0] (4.37, 23.7) circle (0.2cm);
  
  \path[draw=black,line cap=butt,line join=miter,line width=1.pt] 
  (2.7, 21.3).. controls (3.2, 21.8) and (3.8, 22.4) .. (4.36, 22.7)
  (4.36, 22.7).. controls (4.4, 23.2) and (5.5, 25.5) .. (5.9, 26.2)
  (4.6, 27.7).. controls (4.1, 24.5) and (4.4, 22.1) .. (4.6, 19.7);
  
  \path[draw=black,line cap=butt,line join=miter,line width=1.pt,miter limit=4.0,dash pattern=on 0.1cm off 0.1cm] 
  (12.1, 21.4) -- (14.2, 21.4)
  (10.5, 21.4) -- (12.1, 21.4);
  \begin{scriptsize}
  \node[anchor=south west] at (10.3, 28.){$\overline{\alpha}$};
  \node[anchor=south west] at (11.9, 28.){$\overline{\gamma}$};
  \node[anchor=south west] at (14.0, 28.){$\overline{\beta}$};
  \node[anchor=south west] at (2.9, 28.){$\alpha$};
  \node[anchor=south west] at (4.4, 28.){$\gamma$};
  \node[anchor=south west] at (6.0, 28.){$\beta$};
  \node[anchor=south west] at (1.8, 20.9){$\alpha(0)$};
  \node[anchor=south west] at (4.5, 22.4){$\gamma(s)$};
  \node[anchor=south west] at (6.1, 26.){$\beta(t)$};
  \node[anchor=south west] at (9.5, 21.2){$\overline{\alpha}(0)$};
  \node[anchor=south west] at (12.3, 22.9){$\overline{\gamma}(s)$};
  \node[anchor=south west] at (14.4, 25.6){$\overline{\beta}(t)$};
  \end{scriptsize}
\end{tikzpicture}
\caption{On the left, $\alpha$ and $\beta$ are synchronised parallel to $\gamma$. On the right, $\overline{\alpha} = f_{\alpha,\gamma}\circ \alpha$, $\overline{\beta} = f_{\gamma,\beta}\circ \beta$, and $\overline{\gamma} = f_{\alpha,\gamma}\circ \gamma = f_{\gamma,\beta}\circ \gamma$, and $\overline{\gamma}(s)$ is on $[\overline{\alpha}(0),\overline{\beta}(t)]$.}
\label{fig:three parallel lines}
\end{figure}

On the other hand, by \Cref{lem:weak parallel transitive}, $\alpha$ and $\beta$ are weakly parallel. Therefore, for some $t_0\in\mathbb{R}$, the reparametrization $\widetilde{\beta}\colon \mathbb{R}\to X$ given by $\widetilde{\beta}(t) = \beta(t+t_0)$ is synchronised parallel to $\alpha$, by \Cref{lem:flat strip}. We want to prove that $t_0=0$, so indirectly and without loss of generality we will assume that $t_0>0$. This means there exists a $\tau$- and causality-preserving map $f_{\alpha,\widetilde{\beta}}\colon (\alpha(\mathbb{R})\cup \widetilde{\beta}(\mathbb{R}))\to \mathbb{R}^{1,1}$ such that
\begin{align*}
&f_{\alpha,\widetilde{\beta}}(\alpha(t)) = (t,0)\\
&f_{\alpha,\widetilde{\beta}}(\beta(t+t_0)) = (t,d_S([\alpha]\,,[\beta])) \, ,
\end{align*}
which in particular implies that
\[
t - \tau_X(\alpha(0),\beta(t+t_0)) \xrightarrow[]{t\to \infty} 0\,.
\]
Thus, for sufficiently large $t$, he reverse triangle inequality implies
\[
t - \tau_X(\alpha(0),\beta(t)) > t_0/2\,.
\]
This contradicts inequality \eqref{eq:asymptotics1}. Therefore, $\alpha$ and $\beta$ are synchronised parallel lines, and the result follows.
\end{proof}

\bibliographystyle{abbrv}

\bibliography{references}

\begin{thebibliography}{10}

\bibitem{AB05}
S.~B. Alexander and R.~L. Bishop.
\newblock A cone splitting theorem for {A}lexandrov spaces.
\newblock {\em Pacific J. Math.}, 218(1):1--15, 2005.

\bibitem{alexander-bishop2008}
S.~B. Alexander and R.~L. Bishop.
\newblock Lorentz and semi-{R}iemannian spaces with {A}lexandrov curvature bounds.
\newblock {\em Comm. Anal. Geom.}, 16(2):251--282, 2008.

\bibitem{alexander-graf-kunzinger-saemann2023}
S.~B. Alexander, M.~Graf, M.~Kunzinger, and C.~S\"amann.
\newblock Generalized cones as {L}orentzian length spaces: causality, curvature, and singularity theorems.
\newblock {\em Comm. Anal. Geom.}, 31(6):1469--1528, 2023.

\bibitem{BarreraMontesdeOcaSolis2022}
W.~Barrera, L.~Montes~de Oca, and D.~A. Solis.
\newblock Comparison theorems for {L}orentzian length spaces with lower timelike curvature bounds.
\newblock {\em Gen. Relativity Gravitation}, 54(9):Paper No. 107, 32, 2022.

\bibitem{BEMG85b}
J.~K. Beem, P.~E. Ehrlich, S.~Markvorsen, and G.~J. Galloway.
\newblock Decomposition theorems for {L}orentzian manifolds with nonpositive curvature.
\newblock {\em J. Differential Geom.}, 22(1):29--42, 1985.

\bibitem{BEMG85a}
J.~K. Beem, P.~E. Ehrlich, S.~Markvorsen, and G.~J. Galloway.
\newblock A {T}oponogov splitting theorem for {L}orentzian manifolds.
\newblock In {\em Global differential geometry and global analysis 1984 ({B}erlin, 1984)}, volume 1156 of {\em Lecture Notes in Math.}, pages 1--13. Springer, Berlin, 1985.

\bibitem{Ber25}
T.~Beran.
\newblock A {B}onnet–{M}yers rigidity theorem for globally hyperbolic lorentzian length spaces.
\newblock {\em Proceedings of the Royal Society of Edinburgh: Section A Mathematics}, page 1–42, 2025.

\bibitem{beran-kunzinger-rott2024}
T.~Beran, M.~Kunzinger, and F.~Rott.
\newblock On curvature bounds in {L}orentzian length spaces.
\newblock {\em J. Lond. Math. Soc. (2)}, 110(2):Paper No. e12971, 41, 2024.

\bibitem{BNR25}
T.~Beran, L.~Napper, and F.~Rott.
\newblock Alexandrov's patchwork and the {B}onnet-{M}yers theorem for {L}orentzian length spaces.
\newblock {\em Trans. Amer. Math. Soc.}, 378(4):2713--2743, 2025.

\bibitem{BORS23}
T.~Beran, A.~Ohanyan, F.~Rott, and D.~A. Solis.
\newblock The splitting theorem for globally hyperbolic {L}orentzian length spaces with non-negative timelike curvature.
\newblock {\em Lett. Math. Phys.}, 113(2):Paper No. 48, 47, 2023.

\bibitem{BR24}
T.~Beran and F.~Rott.
\newblock Gluing constructions for {L}orentzian length spaces.
\newblock {\em Manuscripta Math.}, 173(1-2):667--710, 2024.

\bibitem{beran-saemann2023}
T.~Beran and C.~S\"amann.
\newblock Hyperbolic angles in {L}orentzian length spaces and timelike curvature bounds.
\newblock {\em J. Lond. Math. Soc. (2)}, 107(5):1823--1880, 2023.

\bibitem{bridson-haefliger1999}
M.~R. Bridson and A.~Haefliger.
\newblock {\em Metric spaces of non-positive curvature}, volume 319 of {\em Grundlehren der mathematischen Wissenschaften [Fundamental Principles of Mathematical Sciences]}.
\newblock Springer-Verlag, Berlin, 1999.

\bibitem{burago-burago-ivanov2001}
D.~Burago, Y.~Burago, and S.~Ivanov.
\newblock {\em A course in metric geometry}, volume~33 of {\em Graduate Studies in Mathematics}.
\newblock American Mathematical Society, Providence, RI, 2001.

\bibitem{CG71}
J.~Cheeger and D.~Gromoll.
\newblock The splitting theorem for manifolds of nonnegative {R}icci curvature.
\newblock {\em J. Differential Geometry}, 6:119--128, 1971/72.

\bibitem{erös-gieger2025}
D.~Erös and S.~Gieger.
\newblock A synthetic {L}orentzian {C}artan-{H}adamard theorem, 2025.
\newblock Preprint: \url{https://arxiv.org/abs/2506.22197}.

\bibitem{Esc88}
J.-H. Eschenburg.
\newblock The splitting theorem for space-times with strong energy condition.
\newblock {\em J. Differential Geom.}, 27(3):477--491, 1988.

\bibitem{Gal89}
G.~J. Galloway.
\newblock The {L}orentzian splitting theorem without the completeness assumption.
\newblock {\em J. Differential Geom.}, 29(2):373--387, 1989.

\bibitem{Gig14}
N.~Gigli.
\newblock An overview of the proof of the splitting theorem in spaces with non-negative {R}icci curvature.
\newblock {\em Anal. Geom. Metr. Spaces}, 2(1):169--213, 2014.

\bibitem{GW71}
D.~Gromoll and J.~A. Wolf.
\newblock Some relations between the metric structure and the algebraic structure of the fundamental group in manifolds of nonpositive curvature.
\newblock {\em Bull. Amer. Math. Soc.}, 77:545--552, 1971.

\bibitem{harris1982}
S.~G. Harris.
\newblock A triangle comparison theorem for {L}orentz manifolds.
\newblock {\em Indiana Univ. Math. J.}, 31(3):289--308, 1982.

\bibitem{Inn82}
N.~Innami.
\newblock Splitting theorems of {R}iemannian manifolds.
\newblock {\em Compositio Math.}, 47(3):237--247, 1982.

\bibitem{jee1984}
D.~J. Jee.
\newblock Gauss-{B}onnet formula for general {L}orentzian surfaces.
\newblock {\em Geom. Dedicata}, 15(2):215--231, 1984.

\bibitem{kunzinger-saemann2018}
M.~Kunzinger and C.~S\"amann.
\newblock Lorentzian length spaces.
\newblock {\em Ann. Global Anal. Geom.}, 54(3):399--447, 2018.

\bibitem{LY72}
H.~B. Lawson, Jr. and S.~T. Yau.
\newblock Compact manifolds of nonpositive curvature.
\newblock {\em J. Differential Geometry}, 7:211--228, 1972.

\bibitem{Mil67}
A.~D. Milka.
\newblock Metric structure of a certain class of spaces that contain straight lines.
\newblock {\em Ukrain. Geometr. Sb.}, (4):43--48, 1967.

\bibitem{New90}
R.~P. A.~C. Newman.
\newblock A proof of the splitting conjecture of {S}.-{T}.\ {Y}au.
\newblock {\em J. Differential Geom.}, 31(1):163--184, 1990.

\bibitem{otsu-shioya1994}
Y.~Otsu and T.~Shioya.
\newblock The {R}iemannian structure of {A}lexandrov spaces.
\newblock {\em J. Differential Geom.}, 39(3):629--658, 1994.

\bibitem{Sch85}
V.~Schroeder.
\newblock A splitting theorem for spaces of nonpositive curvature.
\newblock {\em Invent. Math.}, 79(2):323--327, 1985.

\bibitem{Top59}
V.~A. Toponogov.
\newblock Riemannian spaces containing straight lines.
\newblock {\em Dokl. Akad. Nauk SSSR}, 127:977--979, 1959.

\bibitem{Top64}
V.~A. Toponogov.
\newblock The metric structure of {R}iemannian spaces of non-negative curvature containing straight lines.
\newblock {\em Sibirsk. Mat. \v Z.}, 5:1358--1369, 1964.

\end{thebibliography}

\end{document}